\documentclass[11pt]{amsart}
\usepackage{float}
\usepackage[usenames]{color}
\usepackage{graphicx}
\usepackage{amssymb}
\usepackage{amscd}
\usepackage{makecell}
\theoremstyle{plain}
\newtheorem{thm}{Theorem}[section]
\newtheorem{lemma}[thm]{Lemma}
\newtheorem{cor}[thm]{Corollary}
\newtheorem{conj}[thm]{Conjecture}
\newtheorem{prop}[thm]{Proposition}

\theoremstyle{definition}

\newtheorem{rmk}[thm]{Remark}
\newtheorem{example}[thm]{Example}

\def\dim{\mathop{\hbox {dim}}\nolimits}

\newcommand{\fra}{\mathfrak{a}}

\newcommand{\frg}{\mathfrak{g}}
\newcommand{\frh}{\mathfrak{h}}

\newcommand{\frk}{\mathfrak{k}}
\newcommand{\frl}{\mathfrak{l}}

\newcommand{\frp}{\mathfrak{p}}
\newcommand{\frq}{\mathfrak{q}}

\newcommand{\frt}{\mathfrak{t}}
\newcommand{\fru}{\mathfrak{u}}

\newcommand{\bbC}{\mathbb{C}}

\newcommand{\bbR}{\mathbb{R}}

\newcommand{\bbZ}{\mathbb{Z}}

\newcommand{\caL}{\mathcal{L}}

\newcommand{\caR}{\mathcal{R}}

\newcommand{\be}{\begin {equation}}
\newcommand{\ee}{\end {equation}}
\newcommand{\bp}{\begin {proof}}
\newcommand{\ep}{\end {proof}}

\begin{document}

\title[Dirac series for $E_{6(-14)}$]
{Dirac series for $E_{6(-14)}$}

\author{Lin-Gen Ding}
\address[Ding]{College of Mathematics and Econometrics, Hunan University, Changsha 410082,
China}
%\email{@hnu.edu.cn}

\author{Chao-Ping Dong}
\address[Dong]{Mathematics and Science College, Shanghai Normal University, Shanghai 200234,
P.~R.~China}
%\address[Dong]{Institute of Mathematics,  Hunan University,
%Changsha 410082, China}
\email{chaopindong@163.com}
%\thanks{Dong is supported by NSFC grant 11571097 and Shanghai Gaofeng Project for University Academic Development Program.}

\author{Haian He}
\address[He]{Department of Mathematics, College of Sciences, Shanghai University, Shanghai 200444, P.~R.~China}
\email{hebe.hsinchu@yahoo.com.tw}

\abstract{Up to equivalence, this paper classifies all the irreducible unitary representations with non-zero Dirac cohomology for the simple Lie group $E_{6(-14)}$, which is of Hermitian symmetric type. Each FS-scattered Dirac series of $E_{6(-14)}$ is realized as a composition factor of certain $A_{\frq}(\lambda)$ module.  Along the way, we have also obtained all the fully supported irreducible unitary representations of $E_{6(-14)}$ with integral infinitesimal characters.}
 \endabstract



\maketitle
\section{Introduction}

As a sequel to \cite{DDY}, this article aims to classify the irreducible unitary representations with non-zero Dirac cohomology for the simple Lie group $E_{6(-14)}$, which is of Hermitian symmetric type.

Let $G_{\bbC}$ be a complex connected  simple algebraic group with finite center.
Let $\sigma: G_{\bbC} \to G_{\bbC}$ be a \emph{real form} of $G_{\bbC}$. That is, $\sigma$ is an antiholomorphic Lie group automorphism and $\sigma^2={\rm Id}$. Let $\theta: G_{\bbC}\to G_{\bbC}$ be the involutive algebraic automorphism of $G_{\bbC}$ corresponding to $\sigma$ via Cartan theorem (see Theorem 3.2 of \cite{ALTV}). Put $G=G_{\bbC}^{\sigma}$ as the group of real points. Note that $G$ must be in the Harish-Chandra class \cite{HC}.  Denote by $K_{\bbC}:=G_{\bbC}^{\theta}$, and put $K:=K_{\bbC}^{\sigma}$. Denote by $\frg_0$ the Lie algebra of $G$, and let
$$
\frg_0=\frk_0\oplus \frp_0
$$
be the  Cartan decomposition corresponding to $\theta$ on the Lie algebra level.

Let $H_f=T_f A_f$ be the $\theta$-stable fundamental Cartan subgroup for $G$. Then  $T_f$ is a maximal torus of $K$, and on the Lie algebra level,
$$
\frh_{f, 0}=\frt_{f, 0}\oplus \fra_{f, 0}
$$
is the unique $\theta$-stable fundamental Cartan subalgebra of $\frg_0$.  As usual, we drop the subscripts to stand for the complexified Lie algebras. For example, $\frg=\frg_0\otimes_{\bbR}\bbC$, $\frh_f=\frh_{f, 0}\otimes_{\bbR}\bbC$ and so on. We fix a non-degenerate invariant symmetric bilinear form $B(\cdot, \cdot)$ on $\frg$. Its restrictions to $\frk$, $\frp$, etc., will also be denoted by the same symbol.

Let $l$ be the rank of $\frg$. Fix a positive root system $\Delta^+(\frg, \frh_f)$, and let $\{\zeta_1, \dots, \zeta_l\}$ be the corresponding fundamental weights. Restricting the roots in $\Delta^+(\frg, \frh_f)$ to $\frt_f$ gives a positive root system $\Delta^+(\frg, \frt_f)$, which may no longer be reduced. Note that
$\Delta^+(\frg, \frt_f)$ is a union of compact positive roots $\Delta^+(\frk, \frt_f)$, and non-compact positive roots $\Delta^+(\frp, \frt_f)$. We denote by the half sum of roots in $\Delta^+(\frg, \frt_f)$ (resp., $\Delta^+(\frk, \frt_f)$) by $\rho$ (resp., $\rho_c$). Then $\rho_n:=\rho-\rho_c$
is the half sum of roots in $\Delta^+(\frp, \frt_f)$.
We will denote the Weyl groups corresponding to these root systems by $W(\frg, \frh_f)$, $W(\frg, \frt_f)$, $W(\frk, \frt_f)$.

Fix an orthonormal basis $Z_1, \dots, Z_n$ of $\frp_0$ with respect to
the inner product induced by the form $B(\cdot, \cdot)$. Let $U(\frg)$ be the
universal enveloping algebra of $\frg$ and let $C(\frp)$ be the
Clifford algebra of $\frp$ with respect to $B(\cdot, \cdot)$. The \emph{Dirac operator}
introduced by Parthasarathy is
\begin{equation}
D:=\sum_{i=1}^{n}\, Z_i \otimes Z_i\in U(\frg)\otimes C(\frp).
\end{equation}
Note that $D$ is independent of the choice of the
orthonormal basis $Z_i$ and it is $K$-invariant under the diagonal
action of $K$ given by adjoint actions on both factors.
Moreover, we can embed the Lie algebra $\frk$ diagonally into $U(\frg)\otimes C(\frp)$ via the following map:
$$
X\mapsto X_{\Delta}:=X\otimes 1 + 1 \otimes \sum_{i<j}\frac{1}{2}B(X, [Z_i, Z_j])Z_i Z_j, \quad X\in\frk.
$$
Denote the image of $\frk$ by $\frk_{\Delta}$, and let $\Omega_{\frk_{\Delta}}$ be its Casimir element. Then
\begin{equation}\label{D-square}
D^2=-\Omega_{\frg}\otimes 1 +\Omega_{\frk_{\Delta}} +(\|\rho_c\|^2-\|\rho\|^2)1\otimes 1,
\end{equation}
where $\Omega_{\frg}$ is the Casimir element for $\frg$. A byproduct is Parthasarathy's \emph{Dirac operator inequality} \cite[Lemma 2.5]{P2}, see \eqref{Dirac-inequality},
which effectively detects non-unitarity. Section \ref{sec-Omega3} will offer thousands of such examples. However, it is by no means sufficient for unitarity.

To sharpen the Dirac operator inequality, and to understand the unitary dual $\widehat{G}$ better, Vogan formulated the notion of Dirac cohomology in 1997 \cite{Vog97}.
Indeed, let $\widetilde{K}$ be  the subgroup of $K\times \text{Pin}\,\frp_0$ consisting of all pairs $(k, s)$ such that $\text{Ad}(k)=p(s)$, where $\text{Ad}: K\rightarrow \text{O}(\frp_0)$ is the adjoint action, and $p: \text{Pin}\,\frp_0\rightarrow \text{O}(\frp_0)$ is the pin double covering map. Namely, $\widetilde{K}$ is constructed from the following diagram:
\[
\begin{CD}
\widetilde{K} @>  >  > {\rm Pin}\, \frp_0 \\
@VVV  @VVpV \\
K @>{\rm Ad}>> {\rm O}(\frp_0)
\end{CD}
\]
Let $S_G$ be a spin module for
$C(\frp)$, and let $\pi$ be a
($\frg$, $K$)-module.
Then $D\in U(\frg)\otimes C(\frp)$ acts on $\pi\otimes S_G$. The \textbf{Dirac
cohomology} of a
($\frg$, $K$)-module $\pi$ is defined as the $\widetilde{K}$-module
\begin{equation}\label{def-Dirac-cohomology}
H_D(\pi)=\text{Ker}\, D/ (\text{Im} \, D \cap \text{Ker} D).
\end{equation}
Here we note that $\widetilde{K}$ acts  on $\pi$
through $K$ and on $S_G$ through the pin group
$\text{Pin}\,{\frp_0}$. Moreover, since ${\rm Ad}(k) (Z_1), \dots, {\rm Ad}(k) (Z_n)$ is still an orthonormal basis of $\frp_0$, it follows that $D$ is $\widetilde{K}$ invariant. Therefore, ${\rm  Ker} D$, ${\rm Im} D$, and $H_D(X)$ are once again $\widetilde{K}$ modules.

By setting  the linear functionals on $\frt_f$ to be zero on $\fra_f$, we regard $\frt_f^{*}$ as a subspace of $\frh_f^{*}$. The  Vogan conjecture  was proved by Huang and Pand\v zi\'c in Theorem 2.3 of \cite{HP}. Let us adopt its slight extension to possibly disconnected groups.

\begin{thm}{\rm (Theorem A of \cite{DH})}\label{thm-HP}
Let $\pi$ be an irreducible ($\frg$, $K$)-module.
Assume that the Dirac
cohomology of $\pi$ is nonzero, and let $\gamma\in\frt_f^{*}\subset\frh_f^{*}$ be any highest weight of a $\widetilde{K}$-type  in $H_D(X)$. Then the infinitesimal character $\Lambda$ of $\pi$ is conjugate to
$\gamma+\rho_{c}$ under the action of the Weyl group $W(\frg,\frh_f)$.
\end{thm}

We care the most about the case that $\pi$ is unitary. Then  $D$ is self-adjoint with respect to a natural Hermitian inner product on $\pi\otimes S_G$, $\text{Ker} D \cap \text{Im} D=0$, and
\begin{equation}\label{Dirac-unitary}
H_D(\pi)=\text{Ker}\, D=\text{Ker}\, D^2.
\end{equation}
Note that $D^2$ has non-negative eigenvalue on any $\widetilde{K}$-type of $\pi\otimes S_G$. Using \eqref{D-square}, one can deduce that
\begin{equation}\label{Dirac-inequality}
\|\gamma+\rho_c\|\geq \|\Lambda\|,
\end{equation}
where $\gamma$ is a highest weight of any $\widetilde{K}$-type occurring in $\pi\otimes S_G$. This is Parthasarathy's Dirac operator inequality \cite{P2}. Moreover, by \cite[Theorem 3.5.2]{HP2}, the inequality \eqref{Dirac-inequality} becomes equality on certain $\widetilde{K}$-types of $\pi\otimes S_G$ if and only if $H_D(\pi)$ is non-vanishing.

Dirac cohomology is by definition an invariant for Lie group representations, and a natural problem is: could we classify $\widehat{G}^d$---the set consisting of all the members of $\widehat{G}$ having non-vanishing Dirac cohomology? As coined by Huang, we call members of $\widehat{G}^d$ the \emph{Dirac series} of $G$. In view of \eqref{Dirac-inequality} above,  the Dirac series of $G$  are exactly the members of $\widehat{G}$ on which Dirac inequality becomes equality.

The current paper aims to classify the Dirac series for $E_{6(-14)}$, by which we actually mean the connected simple real exceptional Lie  group $\texttt{E6\_h}$ in \texttt{atlas}, whose Lie algebra is denoted by ${\rm EIII}$ in Knapp \cite{Kn}. Here \texttt{atlas} \cite{At} is a software which computes many aspects of questions pertaining to Lie group representations. In particular, it detects unitarity based on the algorithm due to Adams, van Leeuwen, Trapa and Vogan in \cite{ALTV}. See Section \ref{sec-atlas} for a very brief account. Our main result is stated as follows.

\begin{thm}\label{thm-EIII}
The set $\widehat{\rm E_{6(-14)}}^d$ consists of $31$ FS-scattered representations (see Table \ref{table-EIII-scattered-part}) whose spin-lowest $K$-types are all unitarily small, and $270$  strings of representations (see Tables \ref{table-EIII-string-part-card0}---\ref{table-EIII-string-part-card5-part2}). Moreover,
each FS-scattered representation can be realized as a composition factor of certain $A_{\frq}(\lambda)$ module (see Table \ref{table-FS-Aqlambda}); each spin-lowest $K$-type of any  Dirac series of $E_{6(-14)}$ occurs with multiplicity one.
\end{thm}

In the statements above, the notion of unitarily small (\emph{u-small} for short henceforth) $K$-type was introduced by Salamanca-Riba and Vogan in \cite{SV}, while that of spin-lowest $K$-type came from \cite{D}. They will be recalled  for $E_{6(-14)}$ in Section 3.

The way that we organize the infinitely many Dirac series for $E_{6(-14)}$ is due to Theorem A of \cite{D17}. More precisely, let $(x, \lambda, \nu)$ be the final \texttt{atlas} parameter of any Dirac series $\pi$ of $E_{6(-14)}$. Then $\pi$ is a \textbf{FS-scattered} representation if the \textbf{KGB element} $x$ (see Section \ref{sec-atlas}) is fully supported, i.e., if $\texttt{support} (x)$ equals $[0, 1, \dots, l-1]$. Otherwise, $\pi$ will be merged into a string of representations.

Thanks to Proposition 3.7 of \cite{HPP}, all the irreducible unitary highest/lowest weight modules of $E_{6(-14)}$ are Dirac series. However, the Dirac series will go beyond the unitary highest/lowest weight modules which were classified \cite{EHW,Ja} in the 1980s. See Example \ref{exam-non-hlwt}.

Along the way, we also meet some unitary representations with zero Dirac cohomology, yet their infinitesimal characters are irregular and rather small. It is conceivable that many of them should be unipotent representations. Thus we record them in Table \ref{table-EIII-FS-integral}, and it is worth mentioning the following.

\begin{cor}
All the fully supported irreducible unitary representations of $E_{6(-14)}$ with integral infinitesimal characters are exhausted by Tables \ref{table-EIII-FS-integral} and \ref{table-EIII-scattered-part}.
\end{cor}

Calculations for $E_{6(-14)}$ and those in \cite{DD20,DDY,D19,DW} lead us to make the following speculation, which should imply \cite[Conjecture B]{D17} in view of Lemma 2.2 of \cite{DW}.

\begin{conj}
Let $\Lambda=\sum_{i=1}^{l} n_i\zeta_i\in\frh_f^*$ be the infinitesimal character of any fully supported representation in $\widehat{G}$,  where each $n_i$ is a nonnegative integer or half-integer. Then we must have $n_i=0$, $\frac{1}{2}$, or $1$ for every $1\leq i\leq l$.
\end{conj}

Finally, let us note a slight extension  of \cite[Theorem A]{D17}. Let
\begin{equation}\label{lattice-real-inf-char}
\mathbb{L}\subseteq i\frt_{f, 0}^*+\fra_{f, 0}^*
\end{equation}
 be any lattice which is invariant under the action of the Weyl group $W(\frg, \frh_f)$. Collect the members of $\widehat{G}$ whose infinitesimal characters belong to $\mathbb{L}$ as $\Pi_{\mathrm{u}}(G, \mathbb{L})$. Replacing Theorem 2.3 of \cite{HP} by the current assumption for $\mathbb{L}$, the same proof of \cite[Theorem A]{D17} leads to the following.

\begin{thm}\label{thm-A-extension}
Let $\mathbb{L}$ be as in \eqref{lattice-real-inf-char}.
For all but finitely many exceptions, any member $\pi\in \Pi_{\mathrm{u}}(G, \mathbb{L})$ is cohomologically induced from a member $\pi_{L}$ in $\widehat{L}$ that is in the good range. Here $L$ is the Levi factor of a proper $\theta$-stable parabolic subgroup of $G$.
\end{thm}

As a consequence, by using the translation functor, one can organize the infinite set $\Pi_{\mathrm{u}}(G, \mathbb{L})$ in a similar way as we express the Dirac series for $G$.
For instance, with more effort, one could pin down all the irreducible unitary representations of $E_{6(-14)}$ with integral infinitesimal characters.

The paper is organized as follows. We recall necessary background in Section 2, and review the structure of $E_{6(-14)}$ in Section 3. Then we report the FS-scattered Dirac series of $\rm E_{6(-14)}$ in Section 4, and give a summary of the strings  in Section 5. Examples will be illustrated carefully in Section 6. We check that any FS-scattered Dirac series of $E_{6(-14)}$ can be realized as a composition factor of an $A_{\frq}(\lambda)$ module in Section 7. All the strings are presented in Section 8.

We have built up several \texttt{Mathematica} files to facilitate the using of \texttt{atlas} for the group $E_{6(-14)}$. One of them is available via the link
\begin{verbatim}
https://www.researchgate.net/publication/331629415_EIII-ScatteredPart
\end{verbatim}
The codes there have been explained carefully so that an interested reader can pick them up without much difficulty.

We apologize that we are unable to give a neat presentation as the complex $E_6$ case \cite{D19}. Instead, we have to demonstrate all the Dirac series of $E_{6(-14)}$  in ten clumsy tables (namely, Table \ref{table-EIII-scattered-part}, Tables \ref{table-EIII-string-part-card0}---\ref{table-EIII-string-part-card5-part2}).  To make the reader feel slightly better, we note that there are various links between the Dirac series (hence the tables). For a glimpse, see Remark \ref{rmk-exam-EHW}, Example \ref{exam-series-hwt}, and Example \ref{exam-string-limit}.

\section{Preliminaries}
This section aims to  collect necessary preliminaries.

\subsection{$\texttt{atlas}$ height and lambda norm}\label{sec-lambda-spin}
Let us simply refer to a $\frk$-type by its highest weight $\mu$. Choose a positive root system $(\Delta^+)^\prime(\frg, \frh_f)$ making $\mu+2\rho_c$ dominant. Let $\rho^\prime$ be the half sum of roots in $(\Delta^+)^\prime(\frg, \frh_f)$. After \cite{SV}, we denote by $\lambda_a(\mu)$ the projection of $\mu+2\rho_c- \rho^\prime$ to the dominant Weyl chamber of $(\Delta^+)^\prime(\frg, \frh_f)$. Then
\begin{equation}\label{spin-norm}
\|\mu\|_{\rm{lambda}}:=\|\lambda_a(\mu)\|.
\end{equation}
Here $\|\cdot\|$ is the norm induced from the form $B(\cdot, \cdot)$. It turns out that this number  is  independent of the choice of  $(\Delta^+)^\prime(\frg, \frh_f)$, and it is the \emph{lambda norm} of the $\frk$-type $\mu$ \cite{Vog81}.  The \texttt{atlas} \emph{height} of $\mu$ can be computed via
\begin{equation}\label{atlas-height}
\sum_{\alpha\in (\Delta^+)^\prime(\frg, \frh_f)}\langle \lambda_a(\mu),  \alpha^{\vee}\rangle.
\end{equation}
Here $\langle \cdot, \cdot\rangle$ is the natural pairing between roots and co-roots.

\subsection{Cohomological induction}

Let $\frq= \frl\oplus\fru$ be  a $\theta$-stable parabolic subalgebra of $\frg$ with Levi factor $\frl$ and nilpotent radical $\fru$. Set $L=N_{G}(\frq)$.

Let us arrange that $\Delta(\fru, \frh_f)\subseteq \Delta^{+}(\frg,\frh_f)$.
Set $\Delta^{+}(\frl, \frh_f)=\Delta(\frl,
\frh_f)\cap \Delta^{+}(\frg,\frh_f)$. Let
$\rho^{L}$  denote the half sum of roots in
$\Delta^{+}(\frl,\frh_f)$,  and denote by $\rho(\fru)$ (resp., $\rho(\fru\cap\frp)$) the half sum of roots in $\Delta(\fru,\frh_f)$ (resp., $\Delta(\fru\cap\frp,\frh_f)$). Then
\begin{equation}\label{relations}
\rho=\rho^{L}+\rho(\fru).
\end{equation}

Let $Z$ be an ($\frl$, $L\cap K$)-module. Cohomological induction functors attach to $Z$ certain ($\frg, K$)-modules $\caL_j(Z)$ and $\caR^j(Z)$, where $j$ is a nonnegative integer.
Suppose that $Z$ has infinitesimal character $\lambda_L\in \frh_f^*$. After \cite{KV}, we say that
$Z$ is {\it good} or {\it in good range} if
\begin{equation}\label{def-good}
\mathrm{Re}\,\langle \lambda_L +\rho(\fru),\, \alpha^\vee \rangle >
0, \quad \forall \alpha\in \Delta(\fru, \frh_f).
 \end{equation} We say that $Z$ is {\it weakly good} if
\begin{equation}\label{def-weakly-good}
\mathrm{Re}\, \langle \lambda_L +\rho(\fru),\, \alpha^\vee \rangle
\geq 0, \quad \forall \alpha\in \Delta(\fru, \frh_f).
\end{equation}

\begin{thm}\label{thm-Vogan-coho-ind}
{\rm (\cite{Vog84} Theorems 1.2 and 1.3, or \cite{KV} Theorems 0.50 and 0.51)}
Suppose the admissible
 ($\frl$, $L\cap K$)-module $Z$ is weakly good.  Then we have
\begin{itemize}
\item[(i)] $\caL_j(Z)=\caR^j(Z)=0$ for $j\neq S(:=\emph{\text{dim}}\,(\fru\cap\frk))$.
\item[(ii)] $\caL_S(Z)\cong\caR^S(Z)$ as ($\frg$, $K$)-modules.
\item[(iii)]  if $Z$ is irreducible, then $\caL_S(Z)$ is either zero or an
irreducible ($\frg$, $K$)-module with infinitesimal character $\lambda_L+\rho(\fru)$.
\item[(iv)]
if $Z$ is unitary, then $\caL_S(Z)$, if nonzero, is a unitary ($\frg$, $K$)-module.
\item[(v)] if $Z$ is in good range, then $\caL_S(Z)$ is nonzero, and it is unitary if and only if $Z$ is unitary.
\end{itemize}
\end{thm}

Take $\Lambda\in\frh_f^*$ such that it is dominant for $\Delta^+(\frg, \frh_f)$.
We say that $\Lambda$ is \emph{real} if it belongs to $i\frt_{f, 0}^* +\fra_{f, 0}^*$,
and $\Lambda$ is {\it strongly regular} if
\begin{equation}\label{def-weakly-good}
\langle \Lambda-\rho,\, \alpha^\vee \rangle
\geq 0, \quad \forall \alpha\in \Delta^+(\frg, \frh_f).
\end{equation}

\begin{thm}\label{thm-SR} \emph{(Salamanca-Riba \cite{Sa})}
Let $X$ be an irreducible $(\frg, K)$-module with a strongly regular real infinitesimal character. If $X$ is unitary, then it is an $A_{\frq}(\lambda)$ module in the good range.
\end{thm}

\subsection{The software \texttt{atlas}}\label{sec-atlas}

Let us recall necessary notation from \cite{ALTV} regarding the Langlands parameters in the software \texttt{atlas} \cite{At}. Let $H_{\bbC}$ be a \emph{maximal torus} of $G_{\bbC}$ with Lie algebra $\frh$. That is, $H_{\bbC}$ is a maximal connected abelian subgroup of $G_{\bbC}$ consisting of diagonalizable matrices. Note that $H_{\bbC}$ is complex connected reductive algebraic. Its \emph{character lattice} is the group of algebraic homomorphisms
$$
X^*:={\rm Hom}_{\rm alg} (H_{\bbC}, \bbC^{\times}).
$$
Choose a Borel subgroup $B_{\bbC} \supset H_{\bbC}$.
In \texttt{atlas}, an irreducible $(\frg, K)$-module $\pi$ is parameterized by a \emph{final} parameter $p=(x, \lambda, \nu)$ via the Langlands classification \cite{ALTV}, where $\lambda \in X^*+\rho$, $\nu \in (X^*)^{-\theta}\otimes_{\bbZ}\bbC$, and $x$ is a KGB element. That is, $x$ is an $K_{\bbC}$-orbit of the Borel variety $G_{\bbC}/B_{\bbC}$. In such a case, the infinitesimal character of $\pi$ is
\begin{equation}\label{inf-char}
\frac{1}{2}(1+\theta)\lambda +\nu \in\frh^*.
\end{equation}
Note that the Cartan involution $\theta$ now becomes $\theta_x$---the involution of $x$, which is given by the command \texttt{involution(x)} in \texttt{atlas}.

Some \texttt{atlas} commands that we shall need in the current article will be illustrated in Section \ref{sec-atlas-basic-commands}. For how to do cohomological induction in \texttt{atlas}, the reader may refer to Paul's lecture \cite{P1} or Section 2.4 of \cite{D17}.

\section{The structure of $E_{6(-14)}$}\label{sec-structure}

From now on, we will fix $G$ as the group $\texttt{E6\_h}$ in \texttt{atlas}. This equal rank group is connected, with center $\bbZ/3\bbZ$, but \emph{not} simply connected.
Note that $(G, K)$ is a Hermitian symmetric pair.
The Lie algebra $\frg_0$ of $G$ is labelled as EIII in \cite[Appendix C]{Kn}. We present
the Vogan diagram for $\frg_0$ in Fig.~\ref{Fig-EIII-Vogan}, where $\alpha_1=\frac{1}{2}(1, -1,-1,-1,-1,-1,-1,1)$, $\alpha_2=e_1+e_2$ and $\alpha_i=e_{i-1}-e_{i-2}$ for $3\leq i\leq 6$.  Let $\zeta_1, \dots, \zeta_6\in\frt_f^*$ be the corresponding fundamental weights for $\Delta^+(\frg, \frt_f)$, where $\frt_f\subset \frk$ is the fundamental Cartan subalgebra of $\frg$. The dual space $\frt_f^*$ will be identified with $\frt_f$ under the form $B(\cdot, \cdot)$. Put
\begin{equation}
\zeta:=\zeta_1=(0,0,0,0,0, -\frac{2}{3}, -\frac{2}{3}, \frac{2}{3}).
\end{equation}
Note that
$$
\rho=(0, 1, 2, 3, 4, -4, -4, 4).
$$
We will use $\{\zeta_1, \dots, \zeta_6\}$ as a basis to express the \texttt{atlas} parameters $\lambda$, $\nu$ and the infinitesimal character $\Lambda$. More precisely, in such cases, $[a, b, c, d, e, f]$ will stand for the vector $a\zeta_1+\cdots+f \zeta_6$. For instance, the trivial representation of $E_{6(-14)}$ has infinitesimal character $\rho=[1,1,1,1,1,1]$.

\begin{figure}[H]
\centering
\scalebox{0.6}{\includegraphics{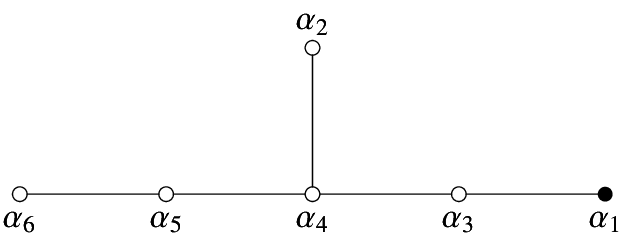}}
\caption{The Vogan diagram for $E_{6(-14)}$}
\label{Fig-EIII-Vogan}
\end{figure}

Denote by $\gamma_i=\alpha_{i+1}$ for $1\leq i\leq 5$.
 Let $\frt_f^-$ be the real linear span of $\gamma_1, \dots, \gamma_5$, which are the simple roots for $\Delta^+(\frk, \frt_f^-)$---the positive system for $\frk$ obtained from $\Delta^+(\frg, \frt_f)$ by restriction. We present the Dynkin diagram of $\Delta^+(\frk, \frt_f^-)$ in Fig.~\ref{Fig-EIII-Dynkin}.
Let $\varpi_1, \dots, \varpi_5\in (\frt_f^-)^*$ be the corresponding fundamental weights.
Note that $\bbR\zeta$ is the one-dimensional center of $\frk$,  that
$$
\frt_f=\frt_f^- \oplus \bbR\zeta,
$$
and that
$$
\rho_c=(0, 1, 2, 3, 4, 0, 0, 0).
$$
Moreover, we have the decomposition
\begin{equation}
\frp=\frp^+ \oplus \frp^-
\end{equation}
as $\frk$-modules, where $\frp^+$ has highest weight
$$
\beta:=\alpha_1+ 2\alpha_2+ 2\alpha_3+ 3\alpha_4+ 2\alpha_5 +\alpha_6=
(\frac{1}{2},\frac{1}{2},\frac{1}{2},\frac{1}{2},
\frac{1}{2},-\frac{1}{2},-\frac{1}{2},\frac{1}{2})
$$
and $\frp^-$ has highest weight $-\alpha_1$. Both of them have dimension $16$. Thus
$$
-\dim \frk + \dim\frp=-46+32=-14.
$$
This is how the number $-14$ enters the name $E_{6(-14)}$, see \cite{He}.

\begin{figure}[H]
\centering
\scalebox{0.6}{\includegraphics{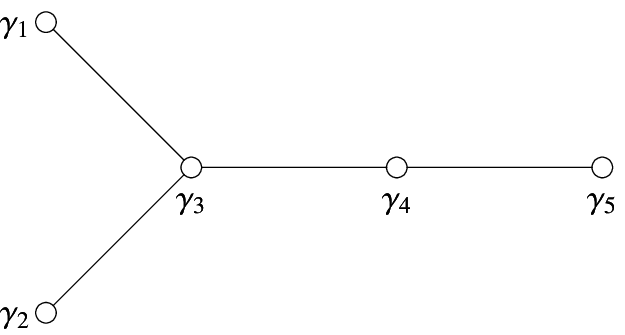}}
\caption{The Dynkin diagram for $\Delta^+(\frk, \frt_f^-)$}
\label{Fig-EIII-Dynkin}
\end{figure}

Let $E_{\mu}$ be the $\frk$-type with highest weight $\mu$.
We will use $\{\varpi_1, \dots, \varpi_5, \frac{1}{4}\zeta\}$ as a basis to express $\mu$.
Namely, in such a case, $[a, b, c, d, e, f]$ stands for the vector $a \varpi_1+b \varpi_2 + c \varpi_3+d \varpi_4 +e \varpi_5+ \frac{f}{4}\zeta$. For instance, $\beta=[1,0,0,0,0,3]$. The $\frk$-type $E_{[a, b, c, d, e, f]}$ has lowest weight $[-b, -a, -c, -d, -e, f]$. We will say that $E_{[a, b, c, d, e, f]}$ and $E_{[b, a, c, d, e, -f]}$ are \emph{dual $\frk$-types}.

For $a, b,c ,d ,e\in\bbZ_{\geq 0}$ and  $f\in\bbZ$, we have that $E_{[a, b, c, d, e, f]}$ is  a $K$-type if and only if
\begin{equation}\label{EIII-K-type}
-\frac{3}{4}a  -\frac{5}{4}b -\frac{3}{2}c -d -\frac{1}{2}e +\frac{1}{4}f \in\bbZ.
\end{equation}
Abuse the notation a bit, we may refer to a $\frk$-type (or $K$-type, $\widetilde{K}$-type) simply by its highest weight in terms of $\{\varpi_1, \dots, \varpi_5, \frac{1}{4}\zeta\}$.

\subsection{Spin norm and spin-lowest $K$-type}
Note that $W(\frg, \frt_f)=51840$ and that $W(\frk, \frt_f)=1920$. Choose from each coset of $W(\frg, \frt_f)/W(\frk, \frt_f)$ the element with minimum length, and collect them as $W(\frg, \frt_f)^1$. Then
$$
|W(\frg, \frt_f)^1|=\frac{|W(\frg, \frt_f)|}{|W(\frk, \frt_f)|}=27.
$$
We enumerate its elements as $w^{(0)}=e, \dots, w^{(26)}$, and put $\rho_n^{(j)}=w^{(j)}\rho-\rho_c$ for $0\leq j \leq 26$. The \emph{spin norm} of the $\frk$-type $E_{\mu}$ introduced  in \cite{D} now specializes as
\begin{equation}\label{spin-norm}
\|\mu\|_{\rm spin}:=\min_{0\leq j \leq 26}\|\{\mu-\rho_n^{(j)}\}+\rho_c\|.
\end{equation}
Here $\{\mu-\rho_n^{(j)}\}$ denotes the unique vector in the dominant chamber for $\Delta^+(\frk, \frt_f)$ to which $\mu-\rho_n^{(j)}$ is conjugate under the action of $W(\frk, \frt_f)$. Note that $\{\mu-\rho_n^{(j)}\}$ is the PRV component \cite{PRV} of the tensor product of $E_{\mu}$ and the dual $\frk$-type of $E_{\rho_n^{(j)}}$. Given a $(\frg, K)$-module $\pi$ of $E_{6(-14)}$, we define its spin norm as
$$
\|\pi\|_{\rm spin}:=\min\|\mu\|_{\rm spin},
$$
where $\mu$ runs over all the $K$-types of $\pi$. We call $\mu$ a \emph{spin-lowest $K$-type} of $\pi$ if it occurs in $\pi$ and that $\|\mu\|_{\rm spin}=\|\pi\|_{\rm spin}$.

When $\pi$ is unitary and has infinitesimal character $\Lambda$, the Dirac operator inequality \eqref{Dirac-inequality} guarantees that
\begin{equation}\label{Dirac-inequality-further}
\|\pi\|_{\rm spin}\geq \|\Lambda\|.
\end{equation}
By Theorem 3.5.2 of \cite{HP2}, the equality happens if and only if
$H_D(\pi)$ is non-vanishing;  moreover, in such a case, it is exactly the spin-lowest $K$-types of $\pi$ that contribute to $H_D(\pi)$.

\subsection{Lambda norm and spin norm of the trivial $K$-type}
This subsection aims to explicitly compute the lambda norm and the spin norm of the trivial $K$-type.
\begin{example}\label{exam-atlas-height-trivial-K-type}
Let $\mu$ be the highest weight of the trivial $K$-type of $E_{6(-14)}$. Then
$$
\mu+2\rho_c=2\rho_c=-10\zeta_1+ 2 \zeta_2 +2 \zeta_3 + 2 \zeta_4+2 \zeta_5+2 \zeta_6
$$
is dominant with respect to the positive root system
$$
(\Delta^+)'(\frg, \frt_f):=s_1s_3s_4s_2s_5s_4s_6 \left(\Delta^+(\frg, \frt_f)\right),
$$
where $s_i$ stands for the simple reflection $s_{\alpha_i}$. Let $\rho'$ be the half sum of the positive roots in $(\Delta^+)'(\frg, \frt_f)$, and let $\zeta_1', \dots, \zeta_6'$ be the corresponding fundamental weights.
Then one computes that
$$
\lambda':=\mu+2\rho_c-\rho'=(-\frac{1}{2},\frac{1}{2},\frac{1}{2},
\frac{3}{2},\frac{3}{2},\frac{1}{2},\frac{1}{2},-\frac{1}{2}).
$$
Moreover, its distance to an arbitrary point $a{\zeta_1'}+\cdots+f{\zeta_6'}$ equals
\begin{align*}
\frac{4}{3}a^2&+2 b^2+\frac{10}{3}c^2+6 d^2+\frac{10 }{3}e^2+\frac{4}{3}f^2+2 a b+\frac{10}{3}a c+4 a d +\frac{8}{3} a e +\frac{4}{3}a f+4 b c+6 b d+4 b e\\
&+2 b f  +8 c d +\frac{16}{3}c e+\frac{8}{3}c f +8 d e  +4 d f + \frac{10}{3}e f-4 a-6 b-6 c-10 d-6 e-4 f+6.
\end{align*}
Under the constraints that $a, \dots, f$ are all non-negative real numbers, one computes that the minimum distance is attained at the point
$$
a=f=\frac{1}{2}, \quad b=1, \quad c=d=e=0.
$$
Then it follows that
$$
\lambda_a(\mu)=(0, \frac{1}{2}, \frac{1}{2}, \frac{3}{2}, \frac{3}{2}, 0, 0, 0)
$$
and that $\mu$ has $\texttt{atlas}$ height $38$.
Moreover,
$$
\|\mu\|_{\rm lambda}=\sqrt{5}, \quad \|\mu\|_{\rm spin}=\|\rho\|=\sqrt{78}.
$$
\end{example}

\subsection{Integral infinitesimal characters}

\begin{prop}\label{prop-EIII-integer-inf-char}
Let $\Lambda$ be the infinitesimal character of any Dirac series of  $E_{6(-14)}$. Then $\Lambda$ must have integer coordinates with respect to the basis $\zeta_1, \dots, \zeta_6$ of $\frt_f^*$.
\end{prop}

\begin{proof}
Take any Dirac series $\pi$ of $E_{6(-14)}$. Let $E_{\mu}$ be any $K$-type of $\pi$ contributing to $H_D(\pi)$. By Theorem \ref{thm-HP}, $\Lambda$ is conjugate to $\{\mu-\rho_n^{(j)}\}+\rho_c$ under the action of $W(\frg, \frt_f)$ for certain $0\leq j\leq 26$. Note that
$$
\{\mu-\rho_n^{(j)}\}=(\mu-\rho_n^{(j)})+\sum_{i=1}^{5}n_i\gamma_i, \quad \mbox{ where } n_i\in\bbZ_{\geq 0}. $$
It is direct to check that the coordinates of any $\rho_n^{(j)}$, or any $\gamma_i$, or $\rho_c$ with respect to $\{\varpi_1, \dots, \varpi_5, \frac{1}{4}\zeta\}$ satisfy \eqref{EIII-K-type}. Since $E_{\mu}$ is assumed to be a $K$-type, it follows that the coordinates $[a, b, c, d, e, f]$ of $\{\mu-\rho_n^{(j)}\}+\rho_c$ meet the requirement \eqref{EIII-K-type}. Now the desired conclusion follows from the fact that
$$
a\varpi_1+\cdots+ e\varpi_5+\frac{f}{4}\zeta=(-\frac{3 a}{4}-\frac{5 b}{4}-\frac{3 c}{2}-d-\frac{e}{2}+\frac{f}{4})\zeta_1+ a\zeta_2 + b\zeta_3+c\zeta_4+d\zeta_5+e\zeta_6.
$$
\end{proof}

Due to the above proposition, for the rest of this article, we always use $\Lambda\in\frt_f^*$ to denote a dominant integral infinitesimal character. That is,
\begin{equation}\label{Lambda}
\Lambda=[a, \dots, f]:=a\zeta_1+\cdots+f\zeta_6 \quad \mbox{ with } a, \dots, f\in \bbZ_{\geq 0}.
\end{equation}

\subsection{KGB elements of $E_{6(-14)}$}\label{sec-KGB}
The group $\texttt{E6\_h}$ has $513$ KGB elements in total, among which $170$ are fully supported. Their indices are listed below.
\begin{align*}
&266, 268, 271, 275, 276, [294, 299], 303, 306, 313, 314, 317, 321, [323, 326], 328, [342, 347],  \\
&349, 351, 354, 360, 362, [365, 369], 372, 373, [375, 392], [394, 396], 398, 400, 402, 403,  \\
&[405, 413], [415, 431], 433, 434, [436, 512].
\end{align*}

For any KGB element $x$, recall that $\theta_x$ is the involution of $x$ given by the command $\texttt{involution(x)}$. Take an arbitrary $\Lambda$ from \eqref{Lambda}.  It is direct to check that whenever $x$ is fully supported,  $\|\Lambda-\theta_x \Lambda\|^2$ is a homogeneous quadratic polynomial in terms of $a, \dots, f$ such that the coefficients of $a^2, \dots, f^2$ are positive, and that the coefficients of $ab, ac, \dots, ef$ are nonnegative. As a direct consequence, we have the following.

\begin{prop}\label{prop-fs-KGB}
Let $x$ be any fully supported KGB element of $\texttt{E6\_h}$. Then there are finitely many $\Lambda$ in \eqref{Lambda} satisfying $\|\Lambda-\theta_x \Lambda\|\leq \|2\rho\|$.
\end{prop}

Indeed, the set
\begin{equation}\label{Omega}
\Omega:=\{\Lambda \mbox{ in } \eqref{Lambda}\mid
\Lambda \mbox{ is not strongly regular and }
\|\Lambda-\theta_x \Lambda\|\leq \|2\rho\| \mbox{ for some f.s. } x\}
\end{equation}
consists of $45696$ elements. Here ``f.s." stands for fully supported.

\subsection{u-small $K$-types}
The notion of u-small $\frk$-type was introduced in \cite{SV}. It has many equivalent characterizations. Let us recall the one stated in Theorem 6.7(d) of \cite{SV} for $E_{6(-14)}$ as follows. See also Lemma 4.4 of \cite{D17}.

\begin{lemma}\label{hyperplane-construction}
The $\frk$-type $E_{\mu}$ is u-small if and only if $\langle\mu+2\rho_c, w^{(j)}\zeta_i
\rangle\leq 2\langle\rho, \zeta_i
\rangle$, $1\leq i\leq 6$, $0\leq j\leq 26$.
\end{lemma}

Based on the above lemma, one enumerates that there are $3153$ u-small $K$-types in total. Let us collect all those $\Lambda$ in \eqref{Lambda} such that $\Lambda$ is conjugate to $\{\mu-\rho_n^{(j)}\}+\rho_c$ under the action of $W(\frg, \frt_f)$ for some $0\leq j\leq 26$ and some u-small $K$-type $E_{\mu}$, and that $\Lambda$ is not strongly regular as $\Omega_2$. It turns out that $\Omega_2$ consists of $1976$ elements and that $\Omega_2\subset \Omega$.

\subsection{A partition of the set $\Omega$} Let us introduce a partial ordering on the set $\Omega$ by setting that $\Lambda_1 \preceq \Lambda_2$ if each coordinate of $\Lambda_2-\Lambda_1$ is nonnegative. Then there are ten minimal elements of the finite poset $(\Omega_2, \preceq)$:
\begin{align*}
&[1, 0, 0, 1, 1, 1], [0, 0, 1, 1, 1, 1], [0, 1, 1, 0, 1, 1],
[1, 1, 0, 1, 0, 1], [1, 1, 1, 0, 1, 0], \\
&[1, 0, 1, 1, 0, 1], [1, 0, 1, 1, 1, 0], [0, 1, 1, 1, 0, 1],
[1, 1, 0, 1, 1, 0], [0, 1, 1, 1, 1, 0].
\end{align*}
Let us collect them as $V$. Put
\begin{equation}\label{Omega3}
\Omega_3:=\{\Lambda\in\Omega\setminus \Omega_2\mid \Lambda_0\preceq\Lambda  \mbox{ for certain } \Lambda_0\in V\},
\end{equation}
and set
\begin{equation}\label{Omega1}
\Omega_1:=\Omega\setminus (\Omega_2\cup \Omega_3).
\end{equation}
Note that $|\Omega_1|=17589$, $|\Omega_3|=26131$, and that we have the following partition of $\Omega$:
\begin{equation}\label{Omega-partition}
\Omega=\Omega_1\cup \Omega_2 \cup \Omega_3.
\end{equation}
\subsection{Some basic $\texttt{atlas}$ commands}\label{sec-atlas-basic-commands}

For any $\Lambda\in\frt_f^*$, collect all the irreducible representations of $E_{6(-14)}$ with infinitesimal character $\Lambda$ as $\Pi(\Lambda)$. By a foundational result of Harish-Chandra, $\Pi(\Lambda)$ must be a finite set. Collect the fully supported members of $\Pi(\Lambda)$ as $\Pi^{\rm fs}(\Lambda)$. Denote by $\Pi_{\rm u}(\Lambda)$ the unitary representations in $\Pi(\Lambda)$, and put $\Pi_{\rm u}^{\rm fs}(\Lambda):=\Pi^{\rm fs}(\Lambda) \cap \Pi_{\rm u}(\Lambda)$.

\begin{example}\label{exam-atlas-basic-EIII}
Let us pin down $\Pi_{\rm u}^{\rm fs}(\Lambda)$ for $\Lambda:=[1,1,1,0,0,0]$. To save space, certain outputs have been omitted.
\begin{verbatim}
G:E6_h
set all=all_parameters_gamma(G, [1,1,1,0,0,0])
#all
Value: 7
for p in all do if is_unitary(p) then prints(p) fi od
final parameter(x=97,lambda=[1,1,1,0,0,0]/1,nu=[1,0,1,-1,0,0]/1)
final parameter(x=52,lambda=[1,1,1,0,0,0]/1,nu=[0,2,0,-1,0,0]/2)
final parameter(x=45,lambda=[1,1,1,0,0,0]/1,nu=[-1,0,2,-1,0,0]/2)
final parameter(x=9,lambda=[1,1,1,0,0,0]/1,nu=[0,0,0,0,0,0]/1)
final parameter(x=0,lambda=[1,1,1,0,0,0]/1,nu=[0,0,0,0,0,0]/1)
\end{verbatim}
The first command sets up the group $\texttt{E6\_h}$ in the memory, while the second one builds up the set $\Pi(\Lambda)$. The third command counts the cardinality of $\Pi(\Lambda)$, which is seven according to its output. The last command tests the unitarity for each representation in $\Pi(\Lambda)$, and prints all the members of $\Pi_{\rm u}(\Lambda)$, which turns out to be five in total. Note that $\Pi_{\rm u}^{\rm fs}(\Lambda)$ is empty since none of the KGB elements $\#97, \#52, \#45, \#9, \#0$ is fully supported, cf. Section \ref{sec-KGB}.
\end{example}

\section{FS-scattered members of $\widehat{E_{6(-14)}}^d$}
This section aims to exhaust the FS-scattered members of $\widehat{E_{6(-14)}}^d$. To do this, assume that $(x, \lambda, \nu)$ is the final $\texttt{atlas}$ parameter of an irreducible unitary representation $\pi$ of $E_{6(-14)}$ where $x$ is fully supported. Let $\Lambda$ as in \eqref{Lambda} be its infinitesimal character (see Proposition \ref{prop-EIII-integer-inf-char}). Then by \cite[Proposition 3.1]{D17}
$$
\nu=\frac{\Lambda-\theta_x(\Lambda)}{2}
$$
must satisfy that $\|\nu\|\leq \|\rho\|$. Moreover, thanks to Theorem \ref{thm-SR}, it suffices to consider the case that $\Lambda$ is not strongly regular. Therefore, $\Lambda$ must belong to the set $\Omega$ defined in \eqref{Omega}. This allows us to enumerate the FS-scattered members of
$\widehat{E_{6(-14)}}^d$ after considering finitely many candidate representations.
To be more precise, it remains to single out the \emph{unitary} ones in
$\bigcup_{\Lambda\in\Omega} \Pi^{\rm fs}(\Lambda)$, and check whether they have Dirac cohomology or not one by one using the method described at step (e) of \cite[Section 8]{D17}.

To make the calculation efficient, we shall use Dirac inequality and the partition \eqref{Omega-partition}. Since $G$ is of Hermitian symmetric type, the $K$-types of $\pi$ may \emph{not} be a union of Vogan pencils (cf. Lemma 3.4 of \cite{Vog80}). See Example \ref{exam-hwt}. Thus unlike \cite{D19} or \cite{DDY}, we can \emph{no longer} use the distribution of the spin norm along Vogan pencils to further improve the computing efficiency.

\subsection{The set $\Omega_3$}\label{sec-Omega3}

Given an arbitrary $\Lambda\in\Omega_3$, we list one \textbf{lowest $K$-types} $\mu$ for every member in $\Pi^{\rm fs}(\Lambda)$. For instance, when $\Lambda=[0, 0, 1, 1, 1, 7]$, this is done in $\texttt{atlas}$ via the command
\begin{verbatim}
for p in all_parameters_gamma(G, [0, 0, 1, 1, 1, 7]) do if
#support(x(p))=6 then prints(highest_weight(LKTs(p)[0], KGB(G,0))) fi od
\end{verbatim}
with the following output
\begin{verbatim}
KGB element #0[  0,  0, -6,  6, -5,  5 ]
KGB element #0[  1,  1, -7,  7, -6,  6 ]
KGB element #0[  0,  1, -7,  7, -5,  6 ]
KGB element #0[  0,  0, -8,  9, -6,  6 ]
KGB element #0[  0,  0,  0,  0,  5, -3 ]
KGB element #0[  1,  0,  1, -1,  5, -3 ]
KGB element #0[  0,  0,  3, -2,  5, -3 ]
KGB element #0[  0,  0,  5, -5,  7, -3 ]
KGB element #0[  0,  1,  4, -4,  6, -3 ]
KGB element #0[  1,  1,  3, -3,  5, -2 ]
KGB element #0[  0,  1,  3, -2,  5, -3 ]
KGB element #0[  1,  0, -3,  3,  4, -3 ]
KGB element #0[  0,  0,  2, -2,  7, -3 ]
KGB element #0[  0,  0, -3,  4,  3, -2 ]
\end{verbatim}
Then we check one by one that
$$
\|\mu\|_{\rm spin}<\|\Lambda\|.
$$
Thus the Dirac inequality always fails, and $\Pi_{\rm u}^{\rm fs}(\Lambda)$ is empty. The same story happens for any other $\Lambda\in\Omega_3$. Thus $\bigcup_{\Lambda\in\Omega_3}\Pi_{\rm u}^{\rm fs}(\Lambda)$ is empty, and the set $\Omega_3$ contributes \emph{no} FS-scattered modules to $\widehat{E_{6(-14)}}^d$.

\subsection{The set $\Omega_1$} We adopt the same method as that for $\Omega_3$. This time the story is no longer that neat: the Dirac inequality fails for $144$ infinitesimal characters in $\Omega_1$. Thus we have to look at them more carefully. Let us illustrate the case $\Lambda=[0,0,1,0,1,1]$.
\begin{verbatim}
for p in all_parameters_gamma(G, [0, 0, 1, 0, 1, 1]) do if
#support(x(p))=6 then prints(p,"    ",is_unitary(p)) fi od
\end{verbatim}
The above command enumerates all the representations in $\Pi(\Lambda)$ whose KGB element $x$ is fully supported. Moreover, it tests the unitarity for them. The output is:
\begin{verbatim}
final parameter(x=429,lambda=[-1,-1,5,-2,3,1]/1,nu=[-1,-1,6,-4,4,0]/2) false
final parameter(x=306,lambda=[0,0,2,-1,1,3]/1,nu=[0,0,3,-3,0,6]/2)  true
\end{verbatim}
Note that the last representation is unitary. Thus $\Pi_{\rm u}^{\rm fs}(\Lambda)$ is \emph{not} empty. However, by Theorem \ref{thm-HP}, the Dirac cohomology of the unique module in $\Pi_{\rm u}^{\rm fs}(\Lambda)$ vanishes since
$$
\|\Lambda\|=\sqrt{\frac{58}{3}}<\|\rho_c\|=\sqrt{30}.
$$

All the members of $\bigcup_{\Lambda\in\Omega_1}\Pi_{\rm u}^{\rm fs}(\Lambda)$ are presented in Table \ref{table-EIII-FS-integral}. These representations are arranged into twelve parts according to their infinitesimal characters.
Using the above argument, one sees that all of them have zero Dirac cohomology. Thus the set $\Omega_1$ contributes \emph{no} FS-scattered modules to $\widehat{E_{6(-14)}}^d$ as well.

\begin{table}
\centering
\caption{Fully supported members of $\widehat{E_{6(-14)}}\setminus \widehat{E_{6(-14)}}^d$ with   integral \protect \\ infinitesimal characters}
\begin{tabular}{rccc}
\Xhline{0.8pt}
$\# x$ &   $\lambda$   & $\nu$ &$\Lambda$  \\
\Xhline{0.65pt}
$306$ & $[0,0,2,-1,1,3]$ & $[0,0,\frac{3}{2},-\frac{3}{2},0,3]$ & $[0, 0, 1, 0, 1, 1]$ \\
\Xcline{1-1}{0.65pt}
$392$ & $[0,1,-2,5,-2,0]$ & $[0,0,-\frac{3}{2},3,-\frac{3}{2},0]$ &$[0, 1, 0, 1, 0, 0]$ \\
\Xcline{1-1}{0.65pt}
$434$ & $[-1,2,4,-1,-1,2]$ & $[-1,1,3,-1,-1,1]$ & $[0, 1, 1, 0, 0, 1]$ \\
$400$ & $[0,3,1,0,-2,5]$ & $[0,\frac{3}{2},0,0,-\frac{3}{2},3]$ & ------ \\
$313$ & $[0,3,2,-1,-1,2]$ &
$[-\frac{1}{2},\frac{5}{2},\frac{3}{2},-\frac{3}{2},-1,\frac{3}{2}]$ & ------ \\
\Xcline{1-1}{0.65pt}
$481$ & $[-1,1,3,-1,3,-1]$ & $[-\frac{3}{2},0,\frac{5}{2},-1,\frac{5}{2},-\frac{3}{2}]$ & $[0, 1, 1, 0, 1, 0]$\\
\Xcline{1-1}{0.65pt}
$385$ & $[3,0,-2,0,3,1]$ & $[2,0,-2,0,2,0]$ & $[1, 0, 0, 0, 1, 1]$ \\
\Xcline{1-1}{0.65pt}
$456$ & $[1,-2,-1,5,-1,1]$ & $[0,-2,-\frac{1}{2},3,-\frac{1}{2},0]$& $[1, 0, 0, 1, 0, 1]$\\
$410$ & $[5,-2,-2,3,0,1]$ & $[3,-1,-2,2,-1,1]$  & ------\\
$408$ & $[1,-2,0,3,-2,5]$ & $[1,-1,-1,2,-2,3]$ & ------\\
$367$ & $[2,-1,-1,4,-3,3]$ & $[\frac{3}{2},-\frac{1}{2},-\frac{3}{2},\frac{5}{2},-\frac{5}{2},2]$ & ------\\
$366$ & $[3,-1,-3,4,-1,2]$ & $[2,-\frac{1}{2},-\frac{5}{2},\frac{5}{2},-\frac{3}{2},\frac{3}{2}]$ & ------\\
$343$ & $[3,0,-2,3,-2,3]$ & $[2,0,-2,2,-2,2]$ & ------\\
$342$ & $[3,0,-2,3,-2,3]$ & $[2,0,-2,2,-2,2]$ & ------\\
\Xcline{1-1}{0.65pt}
$380$ & $[1,0,3,0,-2,3]$ & $[0,0,2,0,-2,2]$ & $[1, 0, 1, 0, 0, 1]$\\
\Xcline{1-1}{0.65pt}
$317$ & $[3,0,1,-1,2,0]$ & $[3,0,0,-\frac{3}{2},\frac{3}{2},0]$ & $[1, 0, 1, 0, 1, 0]$ \\
\Xcline{1-1}{0.65pt}
$326$ & $[3,4,-1,-1,-1,3]$ & $[\frac{3}{2},2,-\frac{1}{2},-1,-\frac{1}{2},\frac{3}{2}]$  & $[1, 1, 0, 0, 0, 1]$\\
\Xcline{1-1}{0.65pt}
$442$ & $[2,2,-1,-1,4,-1]$ & $[1,1,-1,-1,3,-1]$   & $[1, 1, 0, 0, 1, 0]$\\
$417$ & $[5,3,-2,0,1,0]$ & $[3,\frac{3}{2},-\frac{3}{2},0,0,0]$   & ------\\
$321$ & $[2,3,-1,-1,2,0]$ & $[\frac{3}{2},\frac{5}{2},-1,-\frac{3}{2},\frac{3}{2},-\frac{1}{2}]$   & ------\\
\Xcline{1-1}{0.65pt}
$413$ & $[3,1,-2,0,3,1]$ & $[\frac{5}{2},1,-\frac{5}{2},-\frac{1}{2},\frac{5}{2},0]$   & $[1, 1, 0, 0, 1, 1]$\\
$409$ & $[4,1,-1,-1,3,1]$ & $[3,0,-2,-\frac{1}{2},\frac{5}{2},0]$    & ------\\
$376$ & $[2,3,-1,-1,2,1]$ & $[2, 3, -1, -2, 1, 1]$   & ------\\
$323$ & $[2,1,-1,0,1,3]$ & $[2,1,-1,-1,0,3]$   & ------\\
$297$ & $[2,2,0,-1,1,2]$ & $[2,2,0,-2,0,2]$   & ------\\
$295$ & $[2,2,0,-1,1,2]$ & $[2,2,0,-2,0,2]$   & ------\\
\Xcline{1-1}{0.65pt}
$403$ & $[1,3,3,-1,-2,3]$ & $[0,1,\frac{5}{2},-\frac{1}{2},-\frac{5}{2},\frac{5}{2}]$ & $[1, 1, 1, 0, 0, 1]$\\
$394$ & $[1,1,3,-1,-1,4]$ & $[0,0,\frac{5}{2},-\frac{1}{2},-2,3]$ & ------\\
$378$ & $[1,3,2,-1,-1,2]$ & $[1,3,1,-2,-1,2]$ & ------\\
$325$ & $[3,1,1,0,-1,2]$ & $[3,1,0,-1,-1,2]$ & ------\\
$296$ & $[2,2,1,-1,0,2]$ & $[2,2,0,-2,0,2]$ & ------\\
$294$ & $[2,2,1,-1,0,2]$ & $[2,2,0,-2,0,2]$ & ------\\
\Xhline{0.8pt}
\end{tabular}
\label{table-EIII-FS-integral}
\end{table}

\subsection{The set $\Omega_2$}
We adopt the same method as that of $\Omega_3$. This time, for $\Lambda\in\Omega_2$, unlike the case of $\Omega_3$, one could meet non-empty $\Pi_{\rm u}^{\rm fs}(\Lambda)$; moreover, unlike the case of $\Omega_1$, \textbf{all of them} have non-vanishing Dirac cohomology. Some concrete examples will be given in Section \ref{sec-exam}.

\begin{table}
\centering
\caption{FS-scattered Dirac series of $E_{6(-14)}$}
\begin{tabular}{rcccr}
\Xhline{0.8pt}
$\# x$ &   $\lambda$   & $\nu$ & spin LKTs  \\
\Xhline{0.65pt}
$463$ & $[1,2,1,1,-1,3]$ & $[0,\frac{3}{2},2,0,-\frac{7}{2},5]$&
${\rm LKT}=[2,0,0,0,0,-10]$, $[3,0,0,0,0,-7]$\\
& & &  $[2,1,0,0,0,-13]$\\
\cline{4-4}

$412$ & $[3,1,1,1,-1,2]$ & $[1,0,3,0,-\frac{7}{2},\frac{7}{2}]$&
${\rm LKT}=[0, 0, 0, 0, 2, 20]$, $[1, 0, 0, 0, 2, 23]$\\
\cline{4-4}

\Xcline{1-1}{0.65pt}
$502$ & $[2,3,2,-1,1,1]$ & $[1,4,1,-1,0,0]$&
${\rm LKT}=[0, 0, 0, 0, 0, 12]$, $[1, 0, 0, 0, 0, 15]$ \\
& & & $[2, 0, 0, 0, 0, 18]$, $[3, 0, 0, 0, 0, 21]$\\
\cline{4-4}

$496$ & $[1,3,1,0,1,1]$ & $[0,4,0,-1,1,1]$&  ${\rm LKT}=[0, 0, 0, 0, 0, -12]$, $[0, 1, 0, 0, 0, -15]$\\
& & &$[0, 2, 0, 0, 0, -18]$, $[0, 3, 0, 0, 0, -21]$\\
\cline{4-4}

$492$ & $[2,2,2,-1,1,1]$ & $[\frac{3}{2}, \frac{3}{2}, \frac{7}{2},- \frac{7}{2},2,0]$&
${\rm LKT}=[0, 1, 0, 0, 0, 9]$, $[1, 1, 0, 0, 0, 12]$\\
&&& $[0, 3, 0, 0, 0, 3]$, $[2, 1, 0, 0, 0, 15]$\\
\cline{4-4}

$486$ & $[1,2,1,-1,2,2]$ & $[0,\frac{3}{2},2,-\frac{7}{2},\frac{7}{2},\frac{3}{2}]$&
${\rm LKT}=[1, 0, 0, 0, 0, -9]$, $[1, 1, 0, 0, 0, -12]$\\
& & & $[3, 0, 0, 0, 0, -3]$, $[1, 2, 0, 0, 0, -15]$\\
\cline{4-4}

$440$ & $[1,1,2,-1,2,1]$ & $[1,0,3,-\frac{7}{2},\frac{7}{2},0]$&${\rm LKT}=[0, 0, 0, 0, 1, 18]$, $[1, 0, 0, 0, 1, 21]$\\
& & & $[2, 0, 0, 0, 1, 24]$\\
\cline{4-4}

$427$ & $[1,1,2,-1,2,1]$ & $[0,0,\frac{7}{2},-\frac{7}{2},3,1]$&  ${\rm LKT}=[0, 0, 0, 0, 1, -18]$, $[0, 1, 0, 0, 1, -21]$\\
& & &$[0, 2, 0, 0, 1, -24]$\\
\cline{4-4}

$418$ & $[1,3,2,-2,2,1]$ & $[1,4,\frac{3}{2},-4,\frac{3}{2},1]$& ${\rm LKT}=[0, 0, 0, 2, 0, 0]$, $[0, 1, 0, 2, 0, -3]$\\
& & & $[1, 0, 0, 2, 0, 3]$\\
\cline{4-4}

$299$ & $[2,2,1,-1,1,2]$ & $[\frac{5}{2},\frac{5}{2},0,-\frac{5}{2},0,\frac{5}{2}]$&${\rm LKT}=[0, 1, 0, 0, 0, -27]$, $[0, 0, 1, 0, 0, -30]$\\
\cline{4-4}

$298$ & $[2,2,1,-1,1,2]$ & $[\frac{5}{2},\frac{5}{2},0,-\frac{5}{2},0,\frac{5}{2}]$&${\rm LKT}=[1, 0, 0, 0, 0, 27]$, $[0, 0, 1, 0, 0, 30]$\\
\cline{4-4}
\Xcline{1-1}{0.65pt}

$469$ & $[3,2,-1,1,1,1]$ & $[5,\frac{3}{2},-\frac{7}{2},0,2,0]$& ${\rm LKT}=[0, 2, 0, 0, 0, 10]$, $[0, 3, 0, 0, 0, 7]$\\
& & & $[1, 2, 0, 0, 0, 13]$\\
\cline{4-4}

$405$ & $[2,1,-1,1,1,3]$ & $[\frac{7}{2},0,-\frac{7}{2},0,3,1]$&${\rm LKT}=[0, 0, 0, 0, 2, -20]$, $[0, 1, 0, 0, 2, -23]$\\
\cline{4-4}
\Xcline{1-1}{0.65pt}

$493$ & $[1,-1,-1,3,1,2]$ & $[0,-2,-\frac{3}{2},\frac{7}{2},0,\frac{3}{2}]$& $[0, 3, 0, 0, 0, -17]$, $[3,0,0,0,0,1]$\\
\cline{4-4}

$465$ & $[3,-2,-1,3,1,1]$ & $[1,-\frac{7}{2},-\frac{1}{2},\frac{7}{2},0,0]$& ${\rm LKT}=[0,0,0,0,0, 16]$\\
\cline{4-4}
\Xcline{1-1}{0.65pt}

$495$ & $[-1,1,3,-1,2,2]$ & $[-2,0,3,-1,1,2]$&$[3, 0, 0, 0, 0, 5]$, $[0, 3, 0, 0, 0, -13]$\\
\cline{4-4}

$423$ & $[0,1,3,-2,3,1]$ & $[0,0,3,-3,3,0]$& ${\rm LKT}=[0, 0, 0, 0, 0, 20]$\\
\cline{4-4}

$266$ & $[0,2,3,-2,1,2]$ & $[-\frac{3}{2},2,\frac{5}{2},-\frac{5}{2},0,2]$&
$[1, 1, 0, 0, 2, -4]$, $[0, 0, 1, 0, 2, -10]$\\
\cline{4-4}
\Xcline{1-1}{0.65pt}

$499$ & $[3,1,0,1,0,3]$ & $[\frac{5}{2},0,0,0,0,\frac{5}{2}]$&$[0, 3, 0, 0, 0, -9]$, $[3, 0, 0, 0, 0, 9]$\\
\cline{4-4}

$438$ & $[1,1,-1,3,-1,1]$ & $[1,0,-\frac{5}{2},4,-\frac{5}{2},1]$&$[0, 2, 0, 1, 0, -6]$, $[2, 0, 0, 1, 0, 6]$\\
\cline{4-4}

$373$ & $[2,1,-1,2,-1,2]$ & $[\frac{5}{2},1,-\frac{5}{2},2,-\frac{5}{2},\frac{5}{2}]$&$[0, 0, 2, 0, 0, 0]$\\
\cline{4-4}

$369$ & $[3,1,0,1,-1,2]$ & $[3,0,-2,2,-\frac{5}{2},\frac{5}{2}]$&$[0, 1, 0, 0, 0, 21]$\\
\cline{4-4}

$365$ & $[2,1,-1,1,0,3]$ & $[\frac{5}{2},0,-\frac{5}{2},2,-2,3]$&$[1, 0, 0, 0, 0, -21]$\\
\cline{4-4}

$347$ & $[2,1,-1,2,-1,2]$ & $[\frac{5}{2},0,-\frac{5}{2},\frac{5}{2},-\frac{5}{2},\frac{5}{2}]$&${\rm LKT}=[0, 0, 0, 0, 0, -24]$\\
\cline{4-4}

$346$ & $[2,1,-1,2,-1,2]$ & $[\frac{5}{2},0,-\frac{5}{2},\frac{5}{2},-\frac{5}{2},\frac{5}{2}]$&${\rm LKT}=[0, 0, 0, 0, 0, 24]$\\
\cline{4-4}
\Xcline{1-1}{0.65pt}

$501$ & $[3,1,1,0,3,-2]$ & $[2,0,1,-1,3,-2]$& ${\rm LKT}=[0, 3, 0, 0, 0, -5]$, $[3, 0, 0, 0, 0, 13]$\\
\cline{4-4}

$424$ & $[1,1,3,-2,3,0]$ & $[0,0,3,-3,3,0]$& ${\rm LKT}=[0, 0, 0, 0, 0, -20]$\\
\cline{4-4}

$276$ & $[2,2,1,-2,3,0]$ & $[2,2,0,-\frac{5}{2},\frac{5}{2},-\frac{3}{2}]$&$[1, 1, 0, 0, 2, 4]$, $[0, 0, 1, 0, 2, 10]$\\
\cline{4-4}
\Xcline{1-1}{0.65pt}

$500$ & $[2,-1,1,3,-1,1]$ & $[\frac{3}{2},-2,0,\frac{7}{2},-\frac{3}{2},0]$&$[0, 3, 0, 0, 0, -1]$, $[3, 0, 0, 0, 0, 17]$\\
\cline{4-4}

$454$ & $[1,-2,1,3,-1,3]$ & $[0,-\frac{7}{2},0,\frac{7}{2},-\frac{1}{2},1]$&${\rm LKT}=[0, 0, 0, 0, 0, -16]$\\
\cline{4-4}
\Xcline{1-1}{0.65pt}

$512$ & $[1,1,1,1,1,1]$ & $[\frac{5}{2},3,0,0,0,\frac{5}{2}]$& ${\rm LKT}=[0,0,0,0,0,0]$\\
\Xhline{0.8pt}
\end{tabular}
\label{table-EIII-scattered-part}
\end{table}

Summing up the calculations in the three subsections, we conclude that $\bigcup_{\Lambda\in\Omega_3}\Pi_{\rm u}^{\rm fs}(\Lambda)$ are precisely all the FS-scattered members of $\widehat{E_{6(-14)}}^d$. We collect them in Table \ref{table-EIII-scattered-part}.  These representations are arranged into nine parts according to their infinitesimal characters, which are presented below:
\begin{align*}
&[ 1, 1, 1, 1, 0, 1],\quad [1, 1, 1, 0, 1, 1], \quad [ 1, 1, 0, 1, 1, 1 ], \quad [1, 0, 0, 1, 1, 1], \quad [0, 1, 1, 0, 1, 1], \\
&[1, 1, 0, 1, 0, 1], \quad [1, 1, 1, 0, 1, 0], \quad [1, 0, 1, 1, 0, 1], \quad [1,1,1,1,1,1].
\end{align*}
Note that the last column of Table \ref{table-EIII-scattered-part} presents \emph{all} the  spin-lowest $K$-types, which always occur with multiplicity one, of the corresponding representation. Note further that the forth entry in Table \ref{table-EIII-scattered-part} is the  Wallach representation (see Example \ref{exam-Wallach} for more details), while the last entry is the trivial representation.

\section{Strings of $\widehat{E_{6(-14)}}^d$}

In this section, we will report all the strings of $\widehat{E_{6(-14)}}^d$ in Tables \ref{table-EIII-string-part-card0}--\ref{table-EIII-string-part-card5-part2}, with $a, b, c, d, e, f$ being nonnegative integers. Sometimes, there are further requirements on $a, b, c, d, e, f$ indicating where the string starts. They will be put in the column ``spin LKTs".

Let us explain the strategy. Recall that in the previous section, we have searched among certain infinitesimal characters to exhaust the FS-scattered representations of $\widehat{E_{6(-14)}}^d$. Any string of $\widehat{E_{6(-14)}}^d$ is cohomologically induced from a  FS-scattered representation of $\widehat{L}^d_{\rm ss}$ tensored with unitary characters of $L$. Here $L$ is the Levi factor of a \emph{proper} $\theta$-stable parabolic subgroup of $G$, and $L_{\rm ss}:=[L, L]$. Therefore, in the search of the previous section, if we enlarge our attention to the set $\Pi_{\rm u}(\Lambda)$ instead of $\Pi_{\rm u}^{\rm fs}(\Lambda)$, then we shall meet some starting members of each string. Thus there are two remaining tasks: to explicitly pin down the places where each string starts, and to show that each string has no hole.

We can only handle the first task case by case. To be more precise, in each table, the last column presents the spin-lowest $K$-types, which always show up exactly once in the corresponding representation, using the basis $\{\varpi_1, \dots, \varpi_5, \frac{1}{4}\zeta\}$. They occur under certain conditions, which are put immediately after each spin-lowest $K$-type.  To give an example, let us read the string with $\#x=170$ in Table \ref{table-EIII-string-part-card4-part2}. When $a=b=1$, the representation has two spin-lowest $K$-types:  $[1, 2, 0, 0, 2, -3]$ and $[0, 2, 0, 1, 1, 0]$. On the other hand, it has a unique spin-lowest $K$-type $[1, 1, 0, 0, 2, -4]$ when $a=0$ and $b=1$. These conditions tell us the  members for each string.

For the second task, we need some argument to show that every representation described in any string is indeed a Dirac series of $E_{6(-14)}$. Example \ref{exam-trans} considers the string with $\#x=220$, the other ones can be handled similarly.

\subsection{The strings with $|{\rm supp}(x)|=0$}

Let us present all the irreducible tempered representations with non-zero Dirac cohomology for $E_{6(-14)}$ in Table \ref{table-EIII-string-part-card0}. Here $0\leq \# x\leq 26$, $\theta_x$ is always the identity matrix, and $\nu$ is always $[0,0,0,0,0,0]$. Moreover,
$\lambda=\Lambda$ is as in \eqref{Lambda}, with further requirements put immediately after each spin-lowest $K$-type. For instance, we interpret ``$a, b, c+d, d+e, e+f\geq 1$" in the first row of Table \ref{table-EIII-string-part-card0} as a short way for
$$
a\geq 1, \quad b \geq 1, \quad c+d \geq 1, \quad d+e \geq 1, \quad e+f\geq 1.
$$

The following result is immediate from Theorem 1.2 of \cite{DD16}.

\begin{cor}
All the $K$-types of $E_{6(-14)}$ whose spin norm equal to their lambda norm are exactly the ones in the second column of Table \ref{table-EIII-string-part-card0}.
\end{cor}

\subsection{The strings with $|{\rm supp}(x)|=1,2,3,4,5$}\label{sec-string-1-5}
They will be presented in Tables \ref{table-EIII-string-part-card1}---\ref{table-EIII-string-part-card5-part2} in Section 7.

\begin{example}\label{exam-trans}
Let us consider the string with $\#x=220$, which has support $[0,2,3,4]$ in \texttt{atlas}. Fixing $b=0$ and $f=1$ gives us the unique starting module of the string
$$
\pi_{0,1}:=(x, [2,-2,-1,3,1,0],[\frac{3}{2},-2,-\frac{1}{2},1,1,-\frac{3}{2}]).
$$
One computes directly that it has infinitesimal character $\Lambda_{0,1}:=[1, \textbf{0}, 0, 1, 1, \textbf{1}]$, and that it is a Dirac series of $E_{6(-14)}$. The module $\pi_{0,1}$ is cohomologically induced from an irreducible unitary module $\pi_{0,1}^L$ (in the way of Theorem 2.5 of \cite{D17}) of $L$ which is weakly good. Indeed, $\pi_{0,1}^L$ is the representation $\texttt{pL}$ below:
\begin{verbatim}
set p=parameter(KGB(G,220),[2,-2,-1,3,1,0]/1,[3,-4,-1,2,2,-3]/2)
set (P, pL)=reduce_good_range(p)
goodness(pL,G)
Value: "Weakly good"
is_unitary(pL)
Value: true
P
Value: ([0,2,3,4],KGB element #220)
Levi(P)
Value: connected quasisplit real group with Lie algebra 'su(3,2).u(1).u(1)'
\end{verbatim}
Now take arbitrary integers $b\geq 0$, $f\geq 1$. Put $\Lambda_{b, f}:=[1, b, 0, 1, 1, f]$
and
\begin{equation}
\pi_{b,f}:=(x, [2,b-2,-1,3,1,f-1],[\frac{3}{2},-2,-\frac{1}{2},1,1,-\frac{3}{2}]).
\end{equation}
Let $\zeta_{b, f-1}$ be the unitary character of $L$ with differential $b \zeta_2+(f-1)\zeta_6$. Since $\Lambda_{b, f}-\Lambda_{0,1}$ is dominant for $\Delta(\fru, \frt_f)$, by Theorem 7.237 of \cite{KV}, we have that
\begin{equation}
\psi_{\Lambda_{b, f}}^{\Lambda_{0, 1}}(\caL_S(\pi_{0,1}^L\otimes \zeta_{b, f-1}))=\caL_S(\pi_{0,1}^L)=\pi_{0,1}.
\end{equation}
Here $\psi_{\Lambda_{b, f}}^{\Lambda_{0, 1}}$ is the translation functor, and $S:=\dim (\fru\cap\frk)$. In particular, it says that $\caL_S(\pi_{0,1}^L\otimes \zeta_{b, f-1})$ is non-zero. By Theorem \ref{thm-Vogan-coho-ind} (iii) and (iv), this module is irreducible and unitary since the inducing module $\pi_{0,1}^L\otimes \zeta_{b, f-1}$ is  weakly good. This gives us the representation $\pi_{b, f}$. Since $\pi_{0,1}$ has non-zero Dirac cohomology, and that
\begin{equation}
H_D(\pi_{0,1}^L\otimes \zeta_{b, f-1})=H_D(\pi_{0,1}^L)\otimes \zeta_{b, f-1},
\end{equation}
it follows from Theorem B of \cite{DH} that $\pi_{b, f}$ has non-zero Dirac cohomology as well.\hfill\qed
\end{example}

To save space, sometimes we will \emph{fold} two strings into one. Let us illustrate how to unfold, say, the first item of Table \ref{table-EIII-string-part-card1}. Information for $\#32$ is complete there, while that for $\#62$ can be recovered from $\#32$ as follows: for the column ``$\lambda/\nu$", we interchange the 1st and 6th, 3rd and 5th coordinates. That is,
$$
(\#62,  [1, b, c, d, e, f], [1, 0, -\frac{1}{2}, 0, 0, 0]).
$$
For the column ``spin LKTs", we should further pass to dual $K$-types. That is,
$$
[b-1, c, d-1, e-1, f-1, 3b+5c+6d+4e+2f+25], \quad b, d, e, f\geq 1.
$$
Whenever a string is folded, we will present its KGB element \textbf{bolded}.

\section{Examples}\label{sec-exam}
This section aims to look at some Dirac series of $E_{6(-14)}$ more closely.

Let us keep the notation for $G:=E_{6(-14)}$ as in Section \ref{sec-structure}. Then $\frq:=\frk+\frp^+$ is a maximal $\theta$-stable parabolic subalgebra of $\frg$. Let $E_{\mu}$ be the $K$-type with highest weight $\mu\in\frt_f^*$. Form the generalized Verma module
$$
N(\mu):=U(\frg)\otimes_{U(\frq)} E_{\mu}.
$$
Let $L(\mu)$ denote the irreducible quotient of $N(\mu)$. Those unitarizable $L(\mu)$ are called (irreducible) \emph{unitary highest weight modules}. They were classified in \cite{EHW, Ja}, and were known to have non-vanishing Dirac cohomology \cite{HPP}. Moreover, the Dirac cohomology of them can be computed via
Theorem 2.2 of \cite {E} and Theorem 7.11 of \cite{HPR}.

As examples, let us recall a family of unitary highest weight modules from  Corollary 12.6 \cite{EHW}. Put
\begin{equation}\label{mu-EHW}
\mu=(0,0,0,0, z, \frac{z+8}{3}, \frac{z+8}{3}, -\frac{z+8}{3}), \mbox{ where } z\in\bbZ_{>0}.
\end{equation}

\begin{example}\label{exam-hwt}
Set $z=1$ in \eqref{mu-EHW}. Then $\mu=[0,0,0,0,1,-18]$ is the lambda-lowest $K$-type of $L(\mu)$. This module has infinitesimal character $\mu+\rho=[-4, 1, 1, 1, 1, 2]$, which is conjugate to $\Lambda:=[1,1,1,0,1,1]$ under the action of $W(\frg, \frt_f)$.

Among all the $207$ irreducible representations in $\Pi(\Lambda)$, one can locate $L(\mu)$ as follows:
\begin{verbatim}
set all=all_parameters_gamma(G, [1,1,1,0,1,1])
all[20]
Value: final parameter(x=427,lambda=[1,1,2,-1,2,1]/1,nu=[0,0,7,-7,6,2]/2)
\end{verbatim}
The last output gives a final parameter of $L(\mu)$.

By Theorem 2.1, to compute $H_D(L(\mu))$, it suffices to look at its $K$-types up to the $\texttt{atlas}$ height $114$. This is realized via the command:
\begin{verbatim}
print_branch_irr_long(all[20], KGB(G,0),114)
(1+0s)*(KGB element #0,[  0,  0, -5,  5, -4,  4 ]) 10     70
(1+0s)*(KGB element #0,[  1,  0, -6,  6, -5,  5 ]) 144    91
(1+0s)*(KGB element #0,[  0,  0, -7,  7, -5,  5 ]) 54     94
(1+0s)*(KGB element #0,[  2,  0, -7,  7, -6,  6 ]) 1050   113
\end{verbatim}
One calculates that they are the following $K$-types in terms of the basis $\{\varpi_1, \dots, \varpi_5, \frac{1}{4}\zeta\}$:
$$
\eta_1:=[0, 0, 0, 0, 1, -18], \, \eta_2:=[0, 1, 0, 0, 1, -21], \, [0, 0, 0, 0, 2, -24], \, \eta_3:=[0, 2, 0, 0, 1, -24].
$$
Their lambda norms and spin norms are
$$
17,\, 27.5, \, 32, \, 41.5,
$$
and
$$
\sqrt{42}, \,\sqrt{42}, \, 5 \sqrt{2}, \,\sqrt{42},
$$
respectively. Note that $\|\Lambda\|=\sqrt{42}$. Thus $\eta_i$, $1\leq i\leq 3$, are precisely all the spin-lowest $K$-types of $L(\mu)$. Then one computes that $H_D(L(\mu))$ consists of the following five $\widetilde{K}$-types without multiplicities:
$$
\quad [0, 0, 0, 0, 0, -12],  \quad [0, 0, 0, 0, 1, \pm 6], \quad [1, 0, 0, 0, 0, 3],
\quad [0, 1, 0, 0, 0, -3].
$$

Note that by Proposition 12.5 of \cite{EHW}, the $K$-types occurring in $L(\mu)$ are exactly the following ones without multiplicities:
$$
[0, x-y, 0, 0, y+1, -3(x+y)-18], \mbox{ with } x\geq y\geq 0.
$$
\end{example}
\begin{rmk}\label{rmk-exam-EHW}
We also obtain the \texttt{atlas} parameters of $L(\mu)$ for other values of $z$. The result is summarized in Table \ref{table-hwt}, where $a\in\bbZ_{\geq 0}$. Note that the first two representations are entries of Table \ref{table-EIII-scattered-part}, while the last one is an entry of Table \ref{table-EIII-string-part-card5-part2}.

\begin{table}[H]
\centering
\caption{A family of highest weight modules}
\begin{tabular}{l|c|c}
$z$ &  ($\# x$,  $\lambda$, $\nu$) &$\Lambda$  \\
\hline
$1$ & ($427$, $[1,1,2,-1,2,1]$, $[0,0,\frac{7}{2},-\frac{7}{2},3,1]$) & $[1,1,1,0,1,1]$\\
$2$ & ($405$, $[2,1,-1,1,1,3]$, $[\frac{7}{2}, 0,-\frac{7}{2}, 0, 3,1]$) & $[1,1,0,1,1,1]$\\
$a+3$ & ($348$, $[a,1,1,1,1,1]$, $[-\frac{7}{2}, 0, 0, 0, 3, 1]$) & $[a,1,1,1,1,1]$
\end{tabular}
\label{table-hwt}
\end{table}
\end{rmk}

\begin{example}\label{exam-Wallach}
In the paper \cite{HPP}, the authors have studied  Dirac cohomology of the Wallach representations for classical groups of Hermitian symmetric type. Note that $E_{6(-14)}$ has a unique Wallach representation, namely  $L(-3\zeta)$, see Theorem 5.2 of \cite{EHW}.
This module has infinitesimal character $\Lambda=[1,1,1,0,1,1]$ as well.
Using the setting of the previous example, it can be located as follows:
\begin{verbatim}
all[2]
Value: final parameter(x=496,lambda=[1,3,1,0,1,1]/1,nu=[0,4,0,-1,1,1]/1)
\end{verbatim}
It appears in Table \ref{table-EIII-scattered-part}.

To compute Dirac cohomology, again it suffices to look at its $K$-types up to the $\texttt{atlas}$ height $114$:
\begin{verbatim}
(1+0s)*(KGB element #0,[  0,  0, -3,  3, -3,  3 ]) 1    46
(1+0s)*(KGB element #0,[  1,  0, -4,  4, -4,  4 ]) 16   67
(1+0s)*(KGB element #0,[  2,  0, -5,  5, -5,  5 ]) 126  88
(1+0s)*(KGB element #0,[  3,  0, -6,  6, -6,  6 ]) 672  110
\end{verbatim}
One calculates that they are the following $K$-types in terms of the basis $\{\varpi_1, \dots, \varpi_5, \frac{1}{4}\zeta\}$:
$$
[0, 0, 0, 0, 0, -12], \quad [0, 1, 0, 0, 0, -15], \quad [0, 2, 0, 0, 0, -18], \quad [0, 3, 0, 0, 0, -21].
$$
Their lambda norms and spin norms are
$$
8,\, 15.5, \, 26, \, 40,
$$
and
$$
\sqrt{42}, \,\sqrt{42}, \, \sqrt{42}, \,\sqrt{42},
$$
respectively. Recall that $\|\Lambda\|=\sqrt{42}$. Thus the above four $K$-types are precisely all the spin-lowest ones of $L(-3\zeta)$. Then we compute that $H_D(L(-3\zeta))$ consists of the following six $\widetilde{K}$-types without multiplicities:
$$
[0,0,0,0,0, \pm 12],  \quad [0,0,0,0,1, \pm 6], \quad [1,0,0,0,0, 3], \quad [0,1,0,0,0, -3].
$$
\end{example}
\begin{rmk}\label{rmk-Wallach}
One calculates that the $\# 502$ module in Table \ref{table-EIII-scattered-part} has the \emph{same} Dirac cohomology as the Wallach module. Let $X$ be the direct sum of the two modules. Enumerate \emph{all} the $W(\frg, \frt)$ translates of $\Lambda=[1,1,1,0,1,1]$ which are dominant regular for $\Delta^+(\frk, \frt)$ (with repetitions) as $\Lambda_1, \dots, \Lambda_{12}$. Then analogous to Theorem 1.2 of \cite{HPP}, we have that
$$
H_D(X)=\bigoplus_{i=1}^{12} E_{\Lambda_i-\rho_c}.
$$
Note that the $\#502$ module in Table \ref{table-EIII-scattered-part} has another parameter which is easier to compare with that of the $\#496$ module:
\begin{verbatim}
set p=parameter (KGB (G)[502],[2,3,2,-1,1,1]/1,[1,4,1,-1,0,0]/1)
set q=parameter (KGB (G)[502],[1,3,1,0,1,1]/1,[1,4,1,-1,0,0]/1)
p=q
Value: true
\end{verbatim}

\end{rmk}

\begin{example}\label{exam-series-hwt}
By Theorem 12.4 (c) of \cite{EHW}, one finds that $L(z \zeta)$ is an irreducible unitary highest weight module of $E_{6(-14)}$ if and only if $z$ is an integer, and that $z=0$ or  $z\in (-\infty, -3]$. We locate these representations in Table \ref{table-hwt-zzeta}, where $f$ runs over the non-negative integers, and the last row gives the label of the table where we can find the representation. For instance, the $\# 26$ representations can be found in Table \ref{table-EIII-string-part-card0}: they are all tempered.

The value $z=-3$ is the first reduction point in the sense of \cite[Theorem 2.4]{EHW}: it gives the Wallach module.
The value $z=0$ produces the trivial representation, while $z=-10, -9, \dots, -4$ are irreducible generalized Verma modules. \hfill\qed
\end{example}

\begin{table}[H]
\caption{Highest weight modules $L(z \zeta)$}
\begin{tabular}{r|c|c|c|c|c|c|c|c|c|c|}
$z$ &   $-f-11$ & $-10$ & $-9$ & $-8$ & $-7$ & $-6$ & $-5$ & $-4$ & $-3$ & $0$\\
\hline
$\#x$ & $\#26$ & $\#32$ & $\#78$ & $\#123$ & $\#238$ & $\#347$ & $\#424$ & $\#454$& $\#496$ & $\#512$\\
\hline
Table label & \ref{table-EIII-string-part-card0} & \ref{table-EIII-string-part-card1} &
\ref{table-EIII-string-part-card2} & \ref{table-EIII-string-part-card3-part2} &
\ref{table-EIII-string-part-card5-part1} & \ref{table-EIII-scattered-part} & \ref{table-EIII-scattered-part} & \ref{table-EIII-scattered-part} & \ref{table-EIII-scattered-part} & \ref{table-EIII-scattered-part}
\end{tabular}
\label{table-hwt-zzeta}
\end{table}

\begin{example}\label{exam-non-hlwt}
Let us look at the representation with $\#x=373$ in Table \ref{table-EIII-scattered-part} more carefully. It has infinitesimal character $\Lambda:=\zeta_1 + \zeta_2 + \zeta_4 + \zeta_6$. Its starting $K$-types are:
\begin{align*}
&{\rm LKT}=[1, 1, 0, 0, 0, 0], [1, 0, 1, 0, 0, -3], [0, 1, 1, 0, 0, 3], [1, 2, 0, 0, 0, -3], [2, 1, 0, 0, 0, 3] \\
& [1, 1, 0, 1, 0, 0], \mbox{ spin LKT}=[0, 0, 2, 0, 0, 0], [0, 2, 1, 0, 0, 0], [2, 0, 1, 0, 0, 0], \dots
\end{align*}
The representation has a unique spin lowest $K$-type which differs from its unique lowest $K$-type. Proposition 3.7 of \cite{HPP}  says that the lowest $K$-type of an irreducible highest/lowest weight module must contribute to the Dirac cohomology.   Therefore, the above representation is neither a highest weight module, nor a lowest weight module. Thus the Dirac series go beyond the classification in \cite{EHW, Ja}.  \hfill\qed
\end{example}

Using the knowledge of infinitesimal character and LKT, we enumerate \emph{all} the non-trivial FS-scattered highest/lowest weight modules in Table \ref{table-FS-hwt-lwt}. The other sixteen non-trivial Dirac series in Table \ref{table-EIII-scattered-part} are neither highest weight modules nor lowest weight modules.

\begin{table}[H]
\caption{FS-scattered highest/lowest weight modules}
\begin{tabular}{r|c|c|c|c|c|c|c|}
hwt & $\#496$ & $\#427$ & $\#299$ & $\#405$ & $\#347$ & $\#424$ & $\#454$ \\
\hline
lwt & $\#502$ & $\#440$ & $\#298$ & $\#412$ & $\#346$ & $\#423$ & $\#465$
\end{tabular}
\label{table-FS-hwt-lwt}
\end{table}

\section{FS-scattered Dirac series and $A_{\frq}(\lambda)$ modules}

This section aims to verify the following statement.

\begin{prop}\label{prop-FS-Aqlambda}
 Any FS-scattered Dirac series of $E_{6(-14)}$ (see Table \ref{table-EIII-scattered-part}) must be a composition factor of an $A_{\frq}(\lambda)$ module.
\end{prop}

Let us start with a few examples.

\begin{example}
Let us look at the string with $\#x=252$ in Table \ref{table-EIII-string-part-card4-part1}. Take $a=0$ and $f=0$.
\begin{verbatim}
set p=parameter(KGB(G,252),[0,1,1,1,1,0],[-5/2,0,0,2,1,-3])
set (P,q)=reduce_good_range(p)
goodness(q,G)
Value: "Weakly good"
Levi(P)
Value: connected real group with Lie algebra 'so*(8)[1,0].u(1).u(1)'
dimension(q)
Value: 1
is_unitary(q)
Value: true
\end{verbatim}
We see that the representation $\texttt{p}$ is a weakly good $A_{\frq}(\lambda)$ module. Now let us move $(a, f)$ from $(0, 0)$ to $(-1, -1)$ and see what happens.
\begin{verbatim}
set qm=parameter(x(q), lambda(q)-[1,0,0,0,0,1], nu(q))
goodness(qm,G)
Value: "Fair"
theta_induce_irreducible(qm,G)
Value:
1*parameter(x=365,lambda=[2,1,-1,1,0,3]/1,nu=[5,0,-5,4,-4,6]/2) [70]
\end{verbatim}
Therefore, the FS-scattered representation with $\#x=365$ in Table \ref{table-EIII-scattered-part} is realized as a fair $A_{\frq}(\lambda)$ module.\hfill\qed
\end{example}

\begin{example}\label{exam-string-limit}
Consider the string with $\#x=348$ (see Table \ref{table-EIII-string-part-card5-part2}). Take $a=1$.
\begin{verbatim}
set p=parameter(KGB(G,348),[1,1,1,1,1,1],[-7/2,0,0,0,3,1])
set (P,q)=reduce_good_range(p)
q=trivial(Levi(P))
Value: true
goodness(q,G)
Value: "Good"
Levi(P)
Value: connected real group with Lie algebra 'so(8,2).u(1)'
\end{verbatim}
We see that $a=1$ gives us an $A_{\frq}(0)$ module, which is actually strongly regular. Now let us move $a$ from $1$ to $-1$.
\begin{verbatim}
dimension(q)
Value: 1
set qm=parameter(x(q), lambda(q)-[2,0,0,0,0,0], nu(q))
goodness(qm,G)
Value: "Fair"
theta_induce_irreducible(qm,G)
Value:
1*parameter(x=405,lambda=[2,1,-1,1,1,3]/1,nu=[7,0,-7,0,6,2]/2) [86]
\end{verbatim}
Therefore, the FS-scattered representation with $\#x=405$ in Table \ref{table-EIII-scattered-part} is realized as a fair $A_{\frq}(\lambda)$ module.
Similar thing happens when $a=-2$.
\begin{verbatim}
set qm=parameter(x(q), lambda(q)-[3,0,0,0,0,0], nu(q))
goodness(qm,G)
Value: "Fair"
theta_induce_irreducible(qm,G)
Value:
1*parameter(x=427,lambda=[1,1,2,-1,2,1]/1,nu=[0,0,7,-7,6,2]/2) [70]
\end{verbatim}
Now let us see what happens when $a=-8$.
\begin{verbatim}
set qm=parameter(x(q), lambda(q)-[9,0,0,0,0,0], nu(q))
goodness(qm,G)
Value: "None"
theta_induce_irreducible(qm,G)
Value:
1*parameter(x=496,lambda=[1,3,1,0,1,1]/1,nu=[0,4,0,-1,1,1]/1) [46]
1*parameter(x=78,lambda=[1,1,1,0,1,1]/1,nu=[0,0,0,-1,1,1]/1) [110]
\end{verbatim}
The Wallach module is realized as a composition factor of an $A_{\frq}(\lambda)$ module, whose other factor is a member of the string with $\#x=78$ (see Table \ref{table-EIII-string-part-card2}) with $(a, b, c, d)=(1, 1, 1, 0)$.\hfill\qed
\end{example}

One way of realizing  each FS-scattered Dirac series of $E_{6(-14)}$ as a composition factor of certain $A_{\frq}(\lambda)$ module is described in Table \ref{table-FS-Aqlambda}, where ``WF" means the $A_{\frq}(\lambda)$ module is weakly fair. Whenever there is a star, it means that the corresponding $A_{\frq}(\lambda)$ module is reducible and we need to pass to a composition factor to obtain the FS-scattered Dirac series. Otherwise, it means that the $A_{\frq}(\lambda)$ module is irreducible and unitary. In particular, Proposition \ref{prop-FS-Aqlambda} follows.

\begin{table}[H]
\centering
\caption{FS-scattered representations and $A_{\frq}(\lambda)$ modules}
\begin{tabular}{l|c|l|c}
$\#x$ &  $A_{\frq}(\lambda)$ & $\#x$ &  $A_{\frq}(\lambda)$\\
\hline
$\#463$ & $\#435$, $f=-1$, Fair   & $\#412$ & $\#371$, $f=-1$, Fair \\
$\#502^*$ & $\#371$, $f=-8$, None & $\#496^*$ & $\#348$, $a=-8$, None \\
$\#492$ & $\#432$, $a=-2$, Fair   & $\#486$ & $\#435$, $f=-2$, Fair \\
$\#440$ & $\#371$, $f=-2$, Fair   & $\#427$ & $\#348$, $a=-2$, Fair \\
$\#418$ & $\#370$, $b=-1$, Fair   & $\#299$ & $\# 244$, $b=-1$, Fair \\
$\#298$ & $\# 243$, $b=-1$, Fair  & $\#469$ & $\#432$, $a=-1$, Fair \\
$\#405$ & $\# 348$, $a=-1$, Fair  & $\#493$ & $\# 435$, $f=-3$, Fair \\
$\#465$ & $\# 371$, $f=-3$, Fair  & $\#495$ & $\# 435$, $f=-4$, Fair \\
$\#423$ & $\# 371$, $f=-4$, Fair  & $\#266$ & $\# 222$, $c=-1$, Fair \\
$\#499$ & $\# 435$, $f=-5$, WF    & $\#438$ & $\# 370$, $b=-2$, Fair \\
$\#373$ & $\# 261$, $a=f=-1$, Fair & $\#369$ & $\# 256$, $a=f=-1$, Fair \\
$\#365$ & $\# 252$, $a=f=-1$, Fair & $\#347$ & $\# 244$, $b=-2$, Fair \\
$\#346$ & $\# 243$, $b=-2$, Fair   &  $\#501$ & $\# 432$, $a=-4$, Fair \\
$\#424$ & $\# 348$, $a=-4$, Fair   &  $\#276$ & $\# 224$, $e=-1$, Fair \\
$\#500$ & $\# 432$, $a=-3$, Fair   &  $\#454$ & $\# 348$, $a=-3$, Fair
\end{tabular}
\label{table-FS-Aqlambda}
\end{table}

\section{Appendix: tables of all the strings of  $\widehat{E_{6(-14)}}^{\mathrm{d}}$}
\label{sec-appendix}

\newpage

\begin{table}[H]
\centering
\caption{Strings of $\widehat{E_{6(-14)}}^{\mathrm{d}}$ with $|{\rm supp}(x)|=0$}
\begin{tabular}{rcc}
\Xhline{0.8pt}
$\# x$ & spin LKT = the unique LKT \\
\Xhline{0.65pt}
$0$ & $[e+f,a-1,c+d,b-1,d+e,a+2c+e-f+3]$, $a, b, c+d, d+e, e+f\geq 1$ \\
$1$ & $[e + f, a + c, d-1, b-1, c + d + e+1, a-c+e-f]$, $b, d, a+c, e+f\geq 1$\\
$2$ & $[d + e + f+1, a-1, c-1, b + d, e-1, a+2c+3d+e-f+6]$, $a, c, e, b+d \geq 1$\\
$3$ & $[f-1, a-1, c + d + e+1, b-1, d-1, a+2c-2e-f]$, $a,b,d,f\geq 1$\\
$4$ & $[e-1, a-1, c + d, b-1, d + e + f+1, a+2c+e+2f+6]$, $a, b, e, c+d\geq 1$\\
$5$ & $[e + f, c-1, d-1, b-1, a + c + d +e+2, -2a-c+e-f-3]$, $b,c,d,e+f\geq 1$\\
$6$ & $[f-1, a + c, d + e, b-1, c + d, a-c-2e-f-3]$, $b,f,a+c,c+d,d+e\geq 1$\\
$7$ & $[e-1, a + c, d-1, b-1, c + d + e + f+2, a-c+e+2f+3]$, $b,d,e,a+c\geq 1$\\
$8$ & $[b + d + e + f+2, a-1, c-1, d-1, e-1, a+3b+2c+3d+e-f+9]$\\
& $a,c,d,e\geq 1$\\

$9$ & $[d + e, a-1, c-1, b + d, e + f, a+2c+3d+e+2f+9]$\\
& $a, c, b+d, d+e, e+f\geq 1$\\

$10$ & $[f-1, c-1, d + e, b-1, a + c + d+1, -2a-c-2e-f-6]$, $b,c,f,d+e\geq 1$\\
$11$ & $[e-1, c-1, d-1, b-1, a + c + d + e + f+3, -2a-c+e+2f]$, $b,c,d,e\geq 1$\\
$12$ & $[f-1, a + c + d+1, e-1, b + d, c-1, a-c-3d-2e-f-6]$, $c,e,f,b+d\geq 1$\\
$13$ & $[b + d + e+1, a-1, c-1, d-1, e + f, a+3b+2c+3d+e+2f+12]$\\
 & $a,c,d,e+f\geq 1$\\

$14$ & $[d-1, a-1, c-1, b + d + e+1, f-1, a+2c+3d+4e+2f+12]$\\
& $a,c,d,f\geq 1$\\
$15$ & $[f-1, c + d, e-1, b + d, a + c, -2a-c-3d-2e-f-9]$\\
 & $e,f,a+c, c+d,b+d\geq 1$\\

$16$ & $[f-1, a + b + c + d+2, e-1, d-1, c-1, a-3b-c-3d-2e-f-9]$\\
 & $c,d,e,f\geq 1$ \\
$17$ & $[b + d, a-1, c-1, d + e, f-1, a+3b+2c+3d+4e+2f+15]$\\
 & $a,c,d,f\geq 1$; or $d=0, a,b,c,e,f\geq 1$\\

$18$ & $[f-1, b + c + d+1, e-1, d-1, a + c, -2a-3b-c-3d-2e-f-12]$\\
 &$d,e,f, a+c\geq 1$\\
$19$ & $[f-1, d-1, e-1, b + c + d+1, a-1, -2a-4c-3d-2e-f-12]$\\
& $a,d,e,f\geq 1$\\
$20$ & $[b-1, a-1, c + d, e-1, f-1, a+3b+2c+6d+4e+2f+18]$\\
 &$a,b,e,f,c+d\geq 1$\\
$21$ & $[f-1, b + d, e-1, c + d, a-1, -2a-3b-4c-3d-2e-f-15]$\\
 & $a,d,e,f\geq 1$; or $d=0, a,b,c,e,f\geq 1$\\
$22$ & $[b-1, a + c, d-1, e-1, f-1, a+3b+5c+6d+4e+2f+21]$\\
 & $b,d,e,f,a+c\geq 1$\\
$23$ & $[f-1, b-1, d + e, c-1, a-1, -2a-3b-4c-6d-2e-f-18]$\\
 &$a,b,c,f, d+e\geq 1$\\
$24$ & $[b-1, c-1, d-1, e-1, f-1, 4a+3b+5c+6d+4e+2f+24]$\\
 & $b,c,d,e,f\geq 1$\\
$25$ & $[e + f, b-1, d-1, c-1, a-1, -2a-3b-4c-6d-5e-f-21]$\\
& $a,b,c,d,e+f\geq 1$\\
$26$ & $[e-1,b-1,d-1,c-1,a-1, -2a-3b-4c-6d-5e-4f-24]$\\
& $a,b,c,d,e\geq 1$\\
\Xhline{0.8pt}
\end{tabular}
\label{table-EIII-string-part-card0}
\end{table}

\begin{table}[H]
\centering
\caption{Strings of $\widehat{E_{6(-14)}}^{\mathrm{d}}$ with $|{\rm supp}(x)|=1$ (folded)}
\begin{tabular}{rcr}
\Xhline{0.8pt}
$\# x$ &   $\lambda/\nu$     & spin LKTs  \\
\Xhline{0.65pt}
$32$ & $[a,  b, c, d, e, 1]$   & $[e, b-1, d-1, c-1, a-1, -2a-3b-4c-6d-5e-25] $\\
$\textbf{62}$ & $[0,0,0,0,-\frac{1}{2}, 1]$ & $a, b, c, d\geq 1$ \\
\cline{1-1}

$31$ & $[a,  b, c, d, e, 1]$   & $[b+d+e+2, a-1, c-1, d-1, e, a+3b+2c+3d+e+11] $\\
$\textbf{61}$ & $[0,0,0,0,-\frac{1}{2}, 1]$ & $a, c, d\geq 1$ \\
\cline{1-1}

$30$ & $[a,  b, c, d, e, 1]$   & $[e, c-1, d-1, b-1, a+c+d+e+3, -2a-c+e-1]$\\
$\textbf{59}$ & $[0,0,0,0,-\frac{1}{2}, 1]$ & $b, c, d\geq 1$ \\
\cline{1-1}

$29$ & $[a,  b, c, d, e, 1]$   & $[e+d+1, a-1, c-1, b+d, e, a+2c+3d+e+8]$\\
$\textbf{60}$ & $[0,0,0,0,-\frac{1}{2}, 1]$ & $a, c, b+d, d+e\geq 1$ \\
\cline{1-1}

$28$ & $[a,  b, c, d, e, 1]$   & $[e, a+c, d-1, b-1, c+d+e+2, a-c+e+2]$\\
$\textbf{57}$ & $[0,0,0,0,-\frac{1}{2}, 1]$ & $b, d, a+c\geq 1$ \\
\cline{1-1}

$27$ & $[a,  b, c, d, e, 1]$   & $[e, a-1, c+d, b-1, d+e+1, a+2c+e+5]$\\
$\textbf{58}$ & $[0,0,0,0,-\frac{1}{2}, 1]$ & $a, b, c+d, d+e\geq 1$ \\
\cline{1-1}

$38$ & $[a,  b, c, d,  1, f]$   & $[f, b-1, d, c-1,  a-1,  -2a-3b-4c-6d-f-23]$\\
$\textbf{50}$ & $[0,0,0,-\frac{1}{2},1,-\frac{1}{2}]$ & $a, b, c,  d+f\geq 1$ \\
\cline{1-1}

$37$ & $[a,  b, c, d,  1, f]$   & $[b+d+1, a-1, c-1, d, f,  a+3b+2c+3d+2f+16]$\\
$\textbf{49}$ & $[0,0,0,-\frac{1}{2},1,-\frac{1}{2}]$ & $a, c, b+d, d+f\geq 1$ \\
\cline{1-1}

$36$ & $[a,  b, c, d,  1, f]$   & $[d, a-1, c-1, b+d+1, f,  a+2c+3d+2f+13]$\\
$\textbf{48}$ & $[0,0,0,-\frac{1}{2},1,-\frac{1}{2}]$ & $a, c, b+d\geq 1$ \\
\cline{1-1}

$35$ & $[a,  b, c, d,  1, f]$   & $[f, c-1, d, b-1, a+c+d+2, -2a-c-f-5]$\\
$\textbf{47}$ & $[0,0,0,-\frac{1}{2},1,-\frac{1}{2}]$ & $ b, c, d+f\geq 1$ \\
\cline{1-1}

$34$ & $[a,  b, c, d,  1, f]$   & $[f, a+c, d, b-1, c+d+1, a-c-f-2]$\\
$\textbf{45}$ & $[0,0,0,-\frac{1}{2},1,-\frac{1}{2}]$ & $ b, a+c, c+d, d+f\geq 1$ \\
\cline{1-1}

$33$ & $[a,  b, c, d,  1, f]$   & $[f, a-1, c+d+1, b-1, d, a+2c-f+1]$\\
$\textbf{46}$ & $[0,0,0,-\frac{1}{2},1,-\frac{1}{2}]$ & $a, b, c+d\geq 1$ \\
\cline{1-1}

$43$ & $[a,  b, c, 1, e, f]$   & $[b, a-1, c, e, f-1, a+3b+2c+4e+2f+21]$\\
$\textbf{44}$ & $[0,-\frac{1}{2},-\frac{1}{2},1,-\frac{1}{2},0]$ & $a, f, b+c, c+e\geq 1$ \\
\cline{1-1}

$40$ & $[a,  b, c, 1, e, f]$   & $[e, a-1, c, b,  e+f+1,  a+2c+e+2f+9]$\\
$\textbf{42}$ & $[0,-\frac{1}{2},-\frac{1}{2},1,-\frac{1}{2},0]$ & $a, b+c, c+e,  e+f\geq 1$ \\
\cline{1-1}

$39$ & $[a,  b, c, 1, e, f]$   & $[e+f+1, a-1, c, b,  e,  a+2c+e-f+6]$\\
$\textbf{41}$ & $[0,-\frac{1}{2},-\frac{1}{2},1,-\frac{1}{2},0]$ & $a, b+c, b+e,  e+f\geq 1$ \\
\cline{1-1}

$54$ & $[a, 1, c, d, e, f]$   & $[d, a-1,  c-1, d+e+1, f-1,  a+2c+3d+4e+2f+15]$\\
$\textbf{56}$ & $[0, 1,0, -\frac{1}{2},0,0]$ & $a, c, f, d+e\geq 1$ \\
\cline{1-1}

$52$ & $[a, 1, c, d, e, f]$   & $[d+e+1, a-1, c-1, d, e+f, a+2c+3d+e+2f+12]$\\
$\textbf{55}$ & $[0, 1,0, -\frac{1}{2},0,0]$ & $a, c, d+e, e+f\geq 1$ \\
\cline{1-1}

$51$ & $[a, 1, c, d, e, f]$   & $[d+e+f+2, a-1, c-1, d, e-1, a+2c+3d+e-f+9]$\\
$\textbf{53}$ & $[0, 1,0, -\frac{1}{2},0,0]$ & $a, c, e\geq 1$ \\

\Xhline{0.8pt}
\end{tabular}
\label{table-EIII-string-part-card1}
\end{table}

\begin{table}[H]
\centering
\caption{Strings of $\widehat{E_{6(-14)}}^{\mathrm{d}}$ with $|{\rm supp}(x)|=2$ (folded)}
\begin{tabular}{lcr}
\Xhline{0.8pt}
$\# x$ &   $\lambda/\nu$     & spin LKTs  \\
\Xhline{0.65pt}

$78$           & $[a,  b, c, d,1,1]$   & $[0, b-1, d, c-1, a-1, -2a-3b-4c-6d-27]$\\
$\textbf{102}$ & $[0,0,0,-1,1,1]$ & $a, b, c\geq 1$ \\
\cline{1-1}

$77$           & $[a,  b, c, d,1,1]$   & $[b+d+2, a-1, c-1, d, 0, a+3b+2c+3d+15]$\\
$\textbf{101}$ & $[0,0,0,-1,1,1]$ & $a, c, b+d\geq 1$ \\
\cline{1-1}

$76$          & $[a,  b, c, d,1,1]$   & $[0, c-1, d, b-1, a+c+d+3, -2a-c-3]$\\
$\textbf{99}$ & $[0,0,0,-1,1,1]$ & $b, c\geq 1$ \\
\cline{1-1}

$75$           & $[a,  b, c, d,1,1]$   & $[d+1, a-1, c-1, b+d+1, 0, a+2c+3d+12]$\\
$\textbf{100}$ & $[0,0,0,-1,1,1]$ & $a, c, b+d\geq 1$ \\
\cline{1-1}

$74$          & $[a,  b, c, d,1,1]$   & $[0, a+c, d, b-1,  c+d+2, a-c]$\\
$\textbf{97}$ & $[0,0,0,-1,1,1]$      & $b, a+c, c+d\geq 1$ \\
\cline{1-1}

$73$          & $[a,  b, c, d,1,1]$   & $[0, a-1, c+d+1, b-1,  d+1, a+2c+3]$\\
$\textbf{98}$ & $[0,0,0,-1,1,1]$ & $a, b, c+d\geq 1$ \\
\cline{1-1}

$63$          & $[a,  b, c, 1,e, 1]$   & $[e+1, a-1,  c,b,e+1,a+2c+e+8]$\\
$\textbf{71}$ & $[0,-\frac{1}{2},-\frac{1}{2},1,-1,1]$  & $a, b+c\geq 1$ \\
\cline{1-1}

$84$          & $[a,  b, c, 1,1, f]$   & $[f, b,  0, c,a-1,-2a-3b-4c-f-26]$\\
$\textbf{89}$ & $[0,-1,-1,1,1,-1]$ & $a, b+f, c+f\geq 1$ \\
\cline{1-1}

$83$          & $[a,  b, c, 1,1, f]$   & $[b+1, a-1, c, 0, f, a+3b+2c+2f+22]$\\
$\textbf{90}$ & $[0,-1,-1,1,1,-1]$ & $a, b+c, c+f\geq 1$ \\
\cline{1-1}

$82$          & $[a,  b, c, 1,1, f]$   & $[f, c, 0, b, a+c+2, -2a-c-f-8]$\\
$\textbf{86}$ & $[0,-1,-1,1,1,-1]$     & $a+c, b+f, c+f\geq 1$ \\
\cline{1-1}

$81$          & $[a,  b, c, 1,1, f]$   & $[0, a-1, c, b+1, f+1, a+2c+2f+13]$\\
$\textbf{88}$ & $[0,-1,-1,1,1,-1]$ & $a, b+c\geq 1$ \\
\cline{1-1}

$80$          & $[a,  b, c, 1,1, f]$   & $[f, a+c+1,  0, b,c+1,a-c-f-5]$\\
$\textbf{85}$ & $[0,-1,-1,1,1,-1]$ & $a+c, b+ c, b+f\geq 1$ \\
\cline{1-1}

$79$          & $[a,  b, c, 1,1, f]$   & $[f+1, a-1,  c+1, b,0,a+2c-f+4]$\\
$\textbf{87}$ & $[0,-1,-1,1,1,-1]$ & $a, b+ c\geq 1$ \\
\cline{1-1}

$64$          & $[a,  b, 1, d, e,1]$                   & $[e, a,  d, b-1,d+e+2,a+e+4]$\\
$\textbf{70}$ & $[-\frac{1}{2},0,1,-\frac{1}{2},-\frac{1}{2},1]$ & $a+d, b, d+e\geq 1$ \\
\cline{1-1}

$65$ & $[a,  b, 1, d, 1, f]$                  & $[f, a,  d+1, b-1,d+1,a-f]$\\
     & $[-\frac{1}{2},0,1,-1,1,-\frac{1}{2}]$ & $b, a+d+f\geq 1$ \\
\cline{1-1}

$66$ & $[a,  1, c, d, e, 1]$   & $[d+e+2, a-1, c-1, d, e, a+2c+3d+e+11]$\\
$\textbf{72}$ & $[0,1,0,-\frac{1}{2},-\frac{1}{2},1]$ & $a, c, d+e\geq 1$ \\
\cline{1-1}

$67$ & $[a,  1, c, d, 1, f]$   & $[d+1, a-1, c-1, d+1, f, a+2c+3d+2f+16]$\\
$\textbf{68}$ & $[0,1,0,-1,1,-\frac{1}{2}]$ & $a, c\geq 1$ \\
\cline{1-1}

$95$ & $[a,  1, c, 1, e, f]$   & $[0, a-1, c, e+1, f-1, a+2c+4e+2f+21]$\\
$\textbf{96}$ & $[0,1,-1,1,-1,0]$ & $a, f,  c+e\geq 1$ \\
\cline{1-1}

$92$ & $[a,  1, c, 1, e, f]$   & $[e+1, a-1, c, 0, e+f+1, a+2c+e+2f+12]$\\
$\textbf{94}$ & $[0,1,-1,1,-1,0]$ & $a, c+e, e+f\geq 1$ \\
\cline{1-1}

$91$ & $[a,  1, c, 1, e, f]$   & $[e+f+2, a-1, c, 0, e, a+2c+e-f+9]$\\
$\textbf{93}$ & $[0,1,-1,1,-1,0]$ & $a, e+f\geq 1$ \\
\cline{1-1}

$69$ & $[1, b, c, d, e, 1]$   & $[e, c, d-1, b-1, c+d+e+3, e-c]$\\
& $[1,0,-\frac{1}{2},0,-\frac{1}{2},1]$ & $b, d\geq 1$ \\
\Xhline{0.8pt}
\end{tabular}
\label{table-EIII-string-part-card2}
\end{table}

\begin{table}[H]
\centering
\caption{Strings of $\widehat{E_{6(-14)}}^{\mathrm{d}}$ with $|{\rm supp}(x)|=3$: part one (folded)}
\begin{tabular}{lcr}
\Xhline{0.8pt}
$\# x$ &   $\lambda/\nu$    & spin LKTs  \\
\Xhline{0.65pt}
$103$ & $[a,  b-1, c-1, 0,3,0]$  & $[1, a-1, c-1, b, 1, a+2c+10]$, $a,b,c\geq 1$ \\
$\textbf{109}$ & $[0,-\frac{1}{2},-\frac{1}{2},0,1,0]$ & $[1, a-1, c, b-1, 1, a+2c+4]$, $a,b,c\geq 1$\\
\cline{1-1}

$105$ & $[a-1,  b-2,0, 3,0, f-1]$  & $[f-1, a, 1,b-1,1, a-f-3]$, $a,b,f\geq 1$ \\
 & $[-\frac{1}{2},-1,0,1,0,-\frac{1}{2}]$ & $[f, a-1, 1, b-1,1,a-f+3]$, $a,b,f\geq 1$\\
\cline{1-1}

$104$ & $[a,  0, c-2, 3,0, f-1]$  & $[1, a-1, c-1,1,f-1, a+2c+2f+18]$, $a,c,f\geq 1$ \\
$\textbf{106}$ & $[0,0,-1,1,0,-\frac{1}{2}]$ & $[1, a-1, c-1,1,f, a+2c+2f+12]$, $a,c,f\geq 1$\\
\cline{1-1}

$120$ & $[a,  b-1, c-1, 2, 0,2]$  & $[0, a-1, c, b+1, 0, a+2c+11]$, $a,b\geq 1$ \\
$\textbf{145}$ & $[0,-1,-1,1,0,1]$ & $[1, a-1, c-1, b+1, 1, a+2c+14]$, $a,c\geq 1$\\
\cline{1-1}

$118$ & $[a,  b-1, c-1, 2, 0,2]$  & $[0, a-1, c+1, b, 0, a+2c+5]$, $a,c\geq 1$ \\
$\textbf{144}$ & $[0,-1,-1,1,0,1]$ & $[1, a-1, c+1, b-1, 1, a+2c+2]$, $a,b\geq 1$\\
\cline{1-1}

$124$ & $[a-1,  b-1,  2, 0,2, f-1]$  & $[f+1, a-1, 1, b-1, 1, a-f+6]$, $a,b\geq 1$ \\
$\textbf{125}$ & $[-1,-1,1,0,1,-1]$ & $[f+1, a, 0, b, 0, a-f+3]$, $b, f\geq 1$\\
\cline{1-1}

$134$ & $[a,  2, c-1, 0, 2, f-1]$  & $[0, a-1, c,0, f, a+2c+2f+19]$, $a,c\geq 1$ \\
$\textbf{141}$ & $[0,1,-1,0,1,-1]$ & $[1, a-1, c-1,1,f-1, a+2c+2f+22]$, $a,c,f\geq 1$\\
\cline{1-1}

$132$ & $[a,  2, c-1, 0, 2, f-1]$  & $[1, a-1, c-1,1, f+1, a+2c+2f+10]$, $a,c\geq 1$ \\
$\textbf{139}$ & $[0,1,-1,0,1,-1]$ & $[0, a-1, c,0,f+1, a+2c+2f+13]$, $a,c,f\geq 1$\\
\Xhline{0.8pt}
\end{tabular}
\label{table-EIII-string-part-card3-part1}
\end{table}

\begin{table}[H]
\centering
\caption{Strings of $\widehat{E_{6(-14)}}^{\mathrm{d}}$ with $|{\rm supp}(x)|=3$: part two (folded)}
\begin{tabular}{lccr}
\Xhline{0.8pt}
$\# x$ &   $\lambda$   & $\nu$ & spin LKTs  \\
\Xhline{0.65pt}

$112$ & $[  1,   b,  c, d,     1,     1]$ & $[1,0,-\frac{1}{2},-1,1,1]$ & $[0, c, d, b-1, c+d+3, -c-2]$\\
$\textbf{114}$ & & & $b\geq 1$, $c+d\geq 1$ \\
\cline{1-1}

$149$ & $[a, b, c, 1, 1, 1]$ & $[0,-\frac{3}{2},-\frac{3}{2},1,1,1]$ & $[0, a-1, c+1, b+1, 0,  a+2c+9]$\\
$\textbf{167}$& & &  $a\geq 1$, $b+c\geq 1$\\
\cline{1-1}

$123$ & $[a, b, c, 1, 1, 1]$ & $[0,-\frac{3}{2},-\frac{3}{2},\frac{3}{2},0,\frac{3}{2}]$ & $[0, b, 0, c, a-1, -2a-3b-4c-30]$\\
$\textbf{147}$ & & & $a\geq 1$\\
\cline{1-1}

$122$ & $[a, b, c, 1, 1, 1]$ & $[0,-\frac{3}{2},-\frac{3}{2},\frac{3}{2},0,\frac{3}{2}]$ & $[b+2, a-1, c, 0, 0, a+3b+2c+21]$\\
$\textbf{146}$& & & $a\geq 1$, $b+c\geq 1$ \\
\cline{1-1}

$121$ & $[a, b, c, 1, 1, 1]$ & $[0,-\frac{3}{2},-\frac{3}{2},\frac{3}{2},0,\frac{3}{2}]$ & $[0, c, 0, b, a+c+3, -2a-c-6]$\\
$\textbf{143}$ & & & $a+c\geq 1$ \\
\cline{1-1}

$119$ & $[a, b, c, 1, 1, 1]$ & $[0,-\frac{3}{2},-\frac{3}{2},\frac{3}{2},0,\frac{3}{2}]$ & $[0, a+c+1, 0, b, c+2, a-c-3]$\\
$\textbf{142}$& & & $a+c\geq 1$, $b+c\geq 1$ \\
\cline{1-1}

$107$ & $[a, b, 1, d, 1, 1]$ & $[-\frac{1}{2},0,1,-\frac{3}{2},1,1]$ & $[0, a,d+1,b-1,d+2,a+2]$ \\
$\textbf{115}$ & & & $b\geq 1$\\
\cline{1-1}

$108$ & $[a, b, 1, 1, e, 1]$ & $[-1,-1,1,1,-\frac{3}{2},1]$ & $[e+1, a, 0, b, e+2, a+e+7]$ \\
$\textbf{113}$ & & & $a+b\geq 1$ \\
\cline{1-1}

$154$ & $[a, b, 1, 1, 1, f]$ & $[-\frac{3}{2},-2,1,1,1,-\frac{3}{2}]$ & $[f+1, a+1,0,b+1,0,a-f]$\\
& & &  $a+b+f\geq 1$ \\
\cline{1-1}

$128$ & $[a, b, 1, 1, 1, f]$ & $[-\frac{3}{2},-\frac{3}{2},\frac{3}{2},0,\frac{3}{2},-\frac{3}{2}]$ & $[b+1,a,0,0,f,a+3b+2f+27]$\\
$\textbf{129}$& & &  $a+b\geq 1$, $a+f\geq 1$ \\
\cline{1-1}

$126$ & $[a, b, 1, 1, 1, f]$ & $[-\frac{3}{2},-\frac{3}{2},\frac{3}{2},0,\frac{3}{2},-\frac{3}{2}]$ & $[0, a,0,b+1,f+2,a+2f+12]$ \\

$\textbf{127}$ &  &  & $a+b\geq 1$ \\
\cline{1-1}

$110$ & $[a, 1, c, d, 1, 1]$ & $[0,1,0,-\frac{3}{2},1,1]$ & $[d+2,a-1,c-1,d+1,0,a+2c+3d+15]$ \\
$\textbf{117}$ & & & $a\geq 1$, $c\geq 1$\\
\cline{1-1}

$111$ & $[a, 1, c, 1, e, 1]$ & $[0,1,-1,1,-\frac{3}{2},1]$ & $[e+2,a-1,c, 0, e+1, a+2c+e+11]$ \\
$\textbf{116}$  & & & $a\geq 1$ \\
\cline{1-1}

$157$ & $[a, 1, c, 1, 1, f]$ & $[0,1,-2,1,1,-\frac{3}{2}]$ & $[0, a-1,c+1,0,f+1, a+2c+2f+19]$\\
$\textbf{160}$& & & $a\geq 1$\\
\cline{1-1}

$135$ & $[a, 1, c, 1, 1, f]$ & $[0,\frac{3}{2},-\frac{3}{2},0,\frac{3}{2},-\frac{3}{2}]$ & $[f, 0, 0, c+1, a-1, -2a-4c-f-26]$\\
$\textbf{140}$ & &  & $a\geq 1$, $c+f\geq 1$ \\
\cline{1-1}

$133$ & $[a, 1, c, 1, 1, f]$ & $[0,\frac{3}{2},-\frac{3}{2},0,\frac{3}{2},-\frac{3}{2}]$ & $[f,c+1, 0, 0, a+c+2, -2a-c-f-11]$\\
$\textbf{137}$ & &  & $a+c\geq 1$, $c+f\geq 1$ \\
\cline{1-1}

$131$ & $[a, 1, c, 1, 1, f]$ & $[0,\frac{3}{2},-\frac{3}{2},0,\frac{3}{2},-\frac{3}{2}]$ & $[f,a+c+2, 0, 0, c+1, a-c-f-8]$\\
$\textbf{136}$ & & & $a+c\geq 1$ \\
\cline{1-1}

$130$ & $[a, 1, c, 1, 1, f]$ & $[0,\frac{3}{2},-\frac{3}{2},0,\frac{3}{2},-\frac{3}{2}]$ & $[f+2,a-1,c+1, 0, 0, a+2c-f+7]$ \\
$\textbf{138}$ &  & &  $a\geq 1$\\
\Xhline{0.8pt}
\end{tabular}
\label{table-EIII-string-part-card3-part2}
\end{table}

\begin{table}[H]
\centering
\caption{Strings of $\widehat{E_{6(-14)}}^{\mathrm{d}}$ with $|{\rm supp}(x)|=4$: part one}
\begin{tabular}{lccr}
\Xhline{0.8pt}
$\# x$ &   $\lambda$   & $\nu$ & spin LKTs  \\
\Xhline{0.65pt}
$255$ & $[a, b, 1, 1, 1, 1]$ & $[-2,-3,1,1,1,1]$ & $[0, a+1,0,b+2,0,a+5]$ \\

$174$ & $[a, b, 1, 1, 1, 1]$ & $[-2,-2,2,0,0,2]$ & $[0, b+1,0,0,a,-2a-3b-31]$ \\

$173$ & $[a, b, 1, 1, 1, 1]$ & $[-2,-2,2,0,0,2]$ & $[b+2,a,0,0,0,a+3b+26]$ \\

$172$ & $[a, b, 1, 1, 1, 1]$ & $[-2,-2,2,0,0,2]$ & $[0,0,0,b+1,a+3,-2a-10]$ \\

$259$ & $[a, 1, c,1,1,1]$ & $[0,1,-3,1,1,1]$ &  $[0, a-1, c+2, 0,0, a+2c+15]$, $a\geq 1$\\

$180$ & $[a, 1, c,1,1,1]$ & $[0, 2,-2,0,0,2]$ &  $[0, 0, 0, c+1,  a-1,  -2a-4c-30]$, $a\geq 1$\\

$178$ & $[a, 1, c,1,1,1]$ &  $[0, 2,-2,0,0,2]$ &  $[0, c+1, 0, 0, a+c+3,  -2a-c-9]$, $a+c\geq 1$ \\

$176$ & $[a, 1, c,1,1,1]$  & $[0, 2,-2,0,0,2]$ &  $[0, a+c+2, 0, 0, c+2, a-c-6]$, $a+c\geq 1$ \\

$159$ & $[a, 1, 1, 1, e, 1]$ & $[-\frac{3}{2},\frac{3}{2},\frac{3}{2},0,-2,1]$ & $[e+2, a, 0, 0, e+2, a+e+10]$ \\

$261$ & $[a, 1, 1, 1, 1, f]$ & $[-\frac{5}{2},1,0,2,0,-\frac{5}{2}]$ & $[f+2, a+2, 0, 0, 0, a-f]$ \\

$256$ & $[a, 1, 1, 1, 1, f]$ & $[-3,0,1,2,0,-\frac{5}{2}]$ & $[0, a+1, 0, 0, f+2, a+2f+21]$ \\

$252$ & $[a, 1, 1, 1, 1, f]$ & $[-\frac{5}{2},0,0,2,1,-3]$ & $[f+1,  0, 0, 0, a+2, -2a-f-21]$ \\

\cline{1-4}
$214$ & $[a-1, 3,1,0,1, f-1]$  & $[-\frac{3}{2},2,1,-1,1,-\frac{3}{2}]$ &  $[f+1, a, 0, 1, 0, a-f+3]$, $f\geq 1$ \\
&  &          &$[f, a+1, 0, 1, 0, a-f-3]$, $a\geq 1$\\
\cline{1-4}

$158$ & $[a-1, 3,0,1,0, f-1]$  & $[-1,2,0,0,0,-1]$  & $[f-1, a+1,0,1,0,  a-f-6]$, $a,f\geq 1$ \\
    &  &       &$[f+1, a-1, 0, 1, 0,  a-f+6]$, $a,f\geq 1$\\
\cline{1-4}

$162$ & $[1, b, c, 1, 1, 1]$ & $[1, -\frac{3}{2},-2,\frac{3}{2},0,\frac{3}{2},]$ & $[0,c+1,0,b,c+3,-c-5]$\\

$166$ & $[1, b,  1, d, 1, 1]$ & $[ 1,0,1,-2,1,1]$ & $[0,0,d+1,b-1,d+3,0]$, $b\geq 1$\\

$165$ & $[1, b,  1, 1, e, 1]$ & $[ \frac{3}{2},-\frac{3}{2},0,\frac{3}{2},-2,1]$ & $[e+1,0,0,b,e+3,e+5]$ \\

$274$ & $[1, b,  1, 1, 1, f]$ & $[1,-3,1,1,1,-2]$ & $[f+1,0,0,b+2,0,-f-5]$\\

$192$ & $[1, b,  1, 1, 1, f]$ & $[2,-2,0,0,2,-2]$ & $[f,b+2,0,0,0,-3b-f-26]$, $b+f\geq 1$\\

$191$ & $[1, b,  1, 1, 1, f]$ & $[2,-2,0,0,2,-2]$ & $[b+1,0,0,0,f,3b+2f+31]$ \\

$189$ & $[1, b,  1, 1, 1, f]$ & $[2,-2,0,0,2,-2]$ & $[0,0,0,b+1,f+3,2f+10]$ \\

$168$ & $[1, 1, c, 1, 1, f]$ & $[1,\frac{3}{2},-2,0,\frac{3}{2},-\frac{3}{2}]$ & $[f, c+2, 0, 0, c+2, -c-f-10]$ \\

$281$ & $[1, 1,  1, 1, e, f]$ & $[1,1,1,1,-3,0]$ & $[f-1,0,e+2,0,0,-2e-f-15]$, $f\geq 1$\\

$198$ & $[1, 1,  1, 1, e, f]$ & $[2,2,0,0,-2,0]$ & $[0,0,0,e+1,f-1,4e+2f+30]$, $f\geq 1$\\

$194$ & $[1, 1,  1, 1, e, f]$ & $[2,2,0,0,-2,0]$ & $[e+1,0,0,0,e+f+3,e+2f+9]$, $e+f\geq 1$\\

$193$ & $[1, 1,  1, 1, e, f]$ & $[2,2,0,0,-2,0]$ & $[e+f+2,0,0,0,e+2,e-f+6]$,  $e+f\geq 1$\\
\Xhline{0.8pt}
\end{tabular}
\label{table-EIII-string-part-card4-part1}
\end{table}

\begin{table}[H]
\centering
\caption{Strings of $\widehat{E_{6(-14)}}^{\mathrm{d}}$ with $|{\rm supp}(x)|=4$: part two (folded)}
\begin{tabular}{lcr}
\Xhline{0.8pt}
$\# x$ &   $\lambda/\nu$    & spin LKTs  \\
\Xhline{0.65pt}

$152$         & $[a-1, b-1, 2,-1,3,0]$  & $[1, a, 0, b-1, 2, a+3]$, $a,b\geq 1$ \\
$\textbf{155}$& $[-1,-1,1, -\frac{1}{2},\frac{3}{2},-\frac{1}{2}]$ & $[0, a-1, 1, b-1,2, a+9]$, $a,b\geq 1$\\
\cline{1-1}

$151$           & $[a-1, b-2, 0, 3,-1,2]$ & $[0, a-1, 1, b-1, 2, a+5]$, $a,b\geq 1$ \\
$\textbf{163}$  & $[-\frac{1}{2},-\frac{3}{2},-\frac{1}{2},\frac{3}{2},-\frac{1}{2},1]$  & $[1, a, 0, b-1,2, a-1]$, $a,b\geq 1$\\
\cline{1-1}

$208$           & $[a,2,c-1,0,1,2]$  & $[0, a-1, c+1, 0,0, a+2c+9]$, $a,c\geq 1$ \\
$\textbf{226}$  & $[0,\frac{3}{2},-\frac{3}{2},-\frac{1}{2}, 1, 1]$ & $[1, a-1, c,1, 0,  a+2c+6]$, $a,c\geq 1$\\
\cline{1-1}

$179$          & $[a,2,c-1,0,1,2]$  & $[1, a-1, c, 0,0, a+2c+18]$, $a,c\geq 1$ \\
$\textbf{197}$ & $[0,\frac{3}{2},-\frac{3}{2},0,0,\frac{3}{2}]$ & $[2, a-1, c-1, 0,1,  a+2c+21]$, $a,c\geq 1$\\
\cline{1-1}

$156$          & $[a, 2, c-1,-1,3,0]$  & $[2, a-1, c-1, 0,1, a+2c+13]$, $a,c\geq 1$ \\
$\textbf{161}$ & $[0,1,-1,-\frac{1}{2},\frac{3}{2},-\frac{1}{2}]$ & $[2, a-1, c-1,  1,0,  a+2c+7]$, $a,c\geq 1$\\
\cline{1-1}

$150$          & $[a, 0, c-2, 3,-1,2]$  & $[2, a-1, c-1, 0, 1, a+2c+17]$, $a,c\geq 1$  \\
$\textbf{164}$ & $[0,-\frac{1}{2},-\frac{3}{2},\frac{3}{2},-\frac{1}{2},1]$ & $[2, a-1, c-1, 1, 0, a+2c+11]$, $a,c\geq 1$ \\
\cline{1-1}

$148$          & $[a-1, 0, 0, 1, 3, f-2]$  & $[f-1, 0, 1, 0, a-1, -2a-f-23]$ \\
$\textbf{153}$ & $[-1,0,0,0,2,-2]$ & $[f-1, 0, 1, 0, a+1, -2a-f-11]$\\
\cline{1-1}

$204$ & $[a-1, b-2, 1,3,-1,2]$ &  $[1, a+1, 0, b, 1, a-2]$, $b\geq 1$ \\
$\textbf{220}$ & $[-\frac{3}{2},-2,1,1,-\frac{1}{2},\frac{3}{2}]$ &
$[0, a+1,0,b+1,0,a+1]$, $a\geq 1$ \\
\cline{1-1}

$202$ & $[a-1, b-1, 2,-1,3,1]$ &  $[0, a, 0, b+1, 0,a+8]$, $b\geq 1$ \\
$\textbf{221}$ & $[-\frac{3}{2},-\frac{3}{2},\frac{3}{2},-\frac{1}{2},1,1]$ &
$[0, a-1,1,b,1,a+11]$, $a, b\geq 1$ \\
\cline{1-1}

$171$ & $[a-1, b-1, 2,1,0,2]$  &  $[0, a, 0, b+1, 1, a+10]$, $b\geq 1$ \\
$\textbf{190}$ & $[-\frac{3}{2},-\frac{3}{2},\frac{3}{2},0,0,\frac{3}{2}]$ &$[0, a-1, 1, b, 2, a+13]$, $a\geq 1$\\
\cline{1-1}

$170$ & $[a-1, b-1, 2,0,1,2]$ &  $[1, a+1, 0, b-1, 2, a-4]$, $b\geq 1$ \\
$\textbf{187}$ & $[-\frac{3}{2},-\frac{3}{2},\frac{3}{2},0,0,\frac{3}{2}]$  &$[0, a+1, 0, b, 1, a-1]$, $a, b\geq 1$ \\
\cline{1-1}

$209$ & $[a, 1, c-2,3,-1,2]$ &  $[0,  a-1, c+1, 0,0,    a+2c+17]$, $a\geq 1$\\
$\textbf{225}$ & $[0,1, -2,1,-\frac{1}{2}, \frac{3}{2}]$  & $[1,a-1,c,0,1,a+2c+20]$, $a, c\geq 1$\\
\cline{1-1}

$175$ & $[a, 2, c-1,1,0,2]$  &  $[2, a-1, c, 1, 0, a+2c+5]$, $a\geq 1$ \\
$\textbf{195}$ & $[0, \frac{3}{2},-\frac{3}{2},0,0,\frac{3}{2}]$  &$[1, a-1, c+1, 0, 0, a+2c+8]$, $a, c\geq 1$ \\
\cline{1-1}

$212$ & $[a-1, 1,1,0,3, f-2]$  &  $[f-1, 0, 1, 0, a, -2a-f-21]$, $f\geq 1$ \\
$\textbf{213}$ & $[-\frac{3}{2},1,1,-1,2,-2]$  &$[ f-1, 0, 1, 0, a+1, -2a-f-15]$, $a, f\geq 1$ \\
\cline{1-1}

$185$ & $[a-1, 2, 2, -1,2, f-1]$  &  $[0, a, 0, 0, f, a+2f+24]$, $a\geq 1$ \\
$\textbf{186}$ & $[ -\frac{3}{2},\frac{3}{2},\frac{3}{2},-\frac{3}{2},\frac{3}{2},-\frac{3}{2}]$  &$[ 0, a-1, 1, 0, f-1, a+2f+27]$, $a, f\geq 1$ \\
\cline{1-1}

$183$ & $[a-1, 2, 2, -1,2, f-1]$  &  $[0, a-1, 1, 0, f+2, a+2f+9]$, $a\geq 1$\\
$\textbf{184}$ & $[ -\frac{3}{2},\frac{3}{2},\frac{3}{2},-\frac{3}{2},\frac{3}{2},-\frac{3}{2}]$  &$[0, a, 0, 0, f+2, a+2f+12]$, $a, f\geq 1$ \\
\cline{1-1}

$181$ & $[a-1, 2, 2, -1,2, f-1]$  &  $[f+2, a-1, 0, 1, 0, a-f+9]$, $a\geq 1$\\
$\textbf{182}$ & $[ -\frac{3}{2},\frac{3}{2},\frac{3}{2},-\frac{3}{2},\frac{3}{2},-\frac{3}{2}]$  &$[f+2, a, 0, 0, 0, a-f+6]$, $f\geq 1$ \\
\Xhline{0.8pt}
\end{tabular}
\label{table-EIII-string-part-card4-part2}
\end{table}

\begin{table}[H]
\centering
\caption{Strings of $\widehat{E_{6(-14)}}^{\mathrm{d}}$ with $|{\rm supp}(x)|=5$: part one (folded)}
\begin{tabular}{lcr}
\Xhline{0.8pt}
$\# x$ &   $\lambda/\nu$    & spin LKTs  \\
\Xhline{0.65pt}

$361$          & $[a-2,2,2,0,1,0]$  & $[0,a-1,0, 2,0,a+15]$, $a\geq 1$\\
$\textbf{358}$ &$[-3,\frac{3}{2},\frac{3}{2},0,0,0]$  & $[0,a-1,2,0,0,a+3]$, $a\geq 1$\\
\cline{1-1}

$359$          & $ [a-2,2,-1,3,-1,2]$  & $[0,a-1,0, 2,0,a+11]$, $a\geq 1$\\
$\textbf{374}$ & $[-2,1,-1,2,-1,1]$  & $[0,a-1,2,0,0,a-1]$, $a\geq 1$\\
\cline{1-1}

$357$          & $ [a-1,1,1,0,2,1]$  & $[0,a,2,0,0,a-4]$\\
 & $[-\frac{5}{2},1,0,-\frac{1}{2},\frac{5}{2},0]$  & $[0,a+1,1,0,0,a-1]$, $a\geq 1$\\
$\textbf{363}$               &  & $[0,a+2,0,0,0,a+2]$, $a\geq 1$\\
\cline{1-1}

$353$ & $ [a-2,-1,2,2,0,1]$  & $[0,a-1,0,2,0,a+19]$, $a\geq 1$\\
$\textbf{364}$ & $[-3,-1,1,2,-1,1]$  & $[0,a-1,2,0,0,a+7]$, $a\geq 1$\\
\cline{1-1}

$352$          & $ [a-1,1,1,0,2,1]$  & $[0,a+1,0,0,0,a+17]$\\
 & $[-3,0,1,-\frac{1}{2},\frac{5}{2},0]$  & $[0,a-1,0,2,0,a+23]$, $a\geq 1$\\
 $\textbf{355}$              &  & $[0,a,0,1,0,a+20]$, $a\geq 1$\\
\cline{1-1}

$312$          & $[a-1,3,1,-1,2,1]$  & $[0,a,0,2,0,a+8]$\\
               & $[-2,3,1,-2,1,1]$   & $[0,a,1,1,0,a+2]$, $a\geq 1$\\
$\textbf{329}$ &                     & $[0,a+1,0,1,0,a+5]$, $a\geq 1$\\
\cline{1-1}

$311$          & $[a-2,1,3,-1,2,1]$  & $[0,a,1,0,0,a+14]$\\
               & $[-3,1,3,-2,1,1]$  & $[0,a-1,1,1,0,a+17]$, $a\geq 1$\\
$\textbf{327}$ &                    & $[0,a-1,2,0,0,a+11]$, $a\geq 1$\\
\cline{1-1}

$309$           & $[a-1,3,1,1,-1,2]$  & $[1,a+1,1,0,0,a-2]$\\
$\textbf{322}$  & $[-\frac{5}{2},1,0,2,-\frac{5}{2},\frac{5}{2}]$  & $[1,a+2,0,0,0,a+1]$, $a\geq 1$\\
\cline{1-1}

$305$           & $[a-2, 1,3,1,-1,2]$  & $[0,a+1,0,0,1,a+19]$\\
$\textbf{315}$  & $[-3,0,1,2,-\frac{5}{2},\frac{5}{2}]$  & $[0, a, 0, 1, 1, a+22]$, $a\geq 1$\\
\cline{1-1}

$301$          & $[a-1,1,1,1,0,3]$  & $[1,0,0,0,a+1,-2a-23]$\\
$\textbf{319}$ & $[-\frac{5}{2},0,0,2,-2,3]$  & $[1,0,0,0,a+2,-2a-17]$, $a\geq 1$\\
\cline{1-1}

$262$           & $[a-1,1,1,0,1,3]$  & $[1,0,0,0,a,-2a-25]$\\
$\textbf{278}$  & $[-2,1,1,-1,0,3]$  & $[1,0,0,0,a+2,-2a-13]$, $a\geq 1$\\
\cline{1-1}

$238$            & $[a-1,2,2,-1,1,2]$  & $[0,0,0,0,a,-2a-28]$\\
$\textbf{249}$   & $[-2,2,2,-2,0,2]$  & $[1,0,0,0,a-1,-2a-31]$, $a\geq 1$\\
\cline{1-1}

$237$          & $[a-1,2,2,-1,1,2]$  & $[1,a,0,0,0,a+23]$, $a\geq 1$\\
$\textbf{250}$ & $[-2,2,2,-2,0,2]$  & $[1,a-1,0,1,0,a+26]$, $a\geq 1$\\
\cline{1-1}

$236$          & $[a-1,2,2,-1,1,2]$ & $[1,0,0,0,a+3,-2a-7]$\\
$\textbf{247}$ & $[-2,2,2,-2,0,2]$   & $[0,0,0,0,a+3,-2a-10]$, $a\geq 1$\\
\cline{1-1}

$234$          & $[a-1,2,2,-1,1,2]$  & $[0, a+1, 1, 0, 1, a-7]$\\
$\textbf{245}$ & $[-2,2,2,-2,0,2]$ & $[0, a+2, 0, 0, 1, a-4]$, $a\geq 1$\\
\cline{1-1}

$211$          & $[a-1,2,2,-2,3,0]$  & $[1, a-1, 0, 1, 1, a+12]$, $a\geq 1$ \\
$\textbf{215}$ & $ [-\frac{3}{2}, \frac{3}{2}, \frac{3}{2},-2,2,-1]$ & $[1, a-1, 1, 0, 1, a+6]$, $a\geq 1$\\
\cline{1-1}

$199$ & $[a-1,0,0,1,1,3]$  & $[1, 0, 0, 0, a-1, -2a-27]$, $a\geq 1$ \\
$\textbf{219}$ & $[-\frac{3}{2},0,0,0,0,3]$ & $[1, 0, 0, 0, a+2, -2a-9]$, $a\geq 1$\\
\cline{1-1}

$318$ &  $[2,b-2,-1,2,2,1]$ & $[0, 0, 0, b+2, 0, 4]$ \\
$\textbf{320}$ &  $[2,-3,-1,1,1,1]$  &  $[0, 1, 0, b+1, 1, 7]$, $b\geq 1$ \\
\cline{1-1}

$270$          &  $[2, b-1,1,-1,3,1]$  & $[0,0,0,b+1,1,6]$, $b\geq 1$ \\
$\textbf{273}$ &  $[2,-2,0,-\frac{1}{2},1, \frac{3}{2}]$ &$[0,1,0,b,2,9]$, $b\geq 1$ \\
\cline{1-1}

$241$ & $[2,b-1,1,1,0, 2]$  & $[ 0,1, 0, b, 3, 11]$\\
$\textbf{242}$ & $[2,-2,0,0,0,2]$ & $[0, 0, 0, b+1, 2, 8]$, $b\geq 1$\\
\cline{1-1}

$205$ & $[0,b-1,3,-1,1,2]$  & $[0,1,0,b-1,3,-1]$, $b\geq 1$ \\
$\textbf{218}$ & $[-1,-\frac{3}{2},2,-\frac{1}{2},0,\frac{3}{2}]$ & $[1,0,0, b-1,3,-7]$, $b\geq 1$ \\

\Xhline{0.8pt}
\end{tabular}
\label{table-EIII-string-part-card5-part1}
\end{table}

\begin{table}[H]
\centering
\caption{Strings of $\widehat{E_{6(-14)}}^{\mathrm{d}}$ with $|{\rm supp}(x)|=5$: part two}
\begin{tabular}{rccr}
\Xhline{0.8pt}
$\# x$ &   $\lambda$   & $\nu$ & spin LKTs  \\
\Xhline{0.65pt}

$432$ & $ [a,1,1,1,1,1]$ & $[-5,\frac{3}{2},\frac{3}{2},0,2,0]$ & $[0,a+3,0, 0,0,a+11]$\\

$348$ & $[a,1,1,1,1,1]$ & $[-\frac{7}{2}, 0, 0, 0, 3, 1]$ & $[0,0,0,0,a+3,-2a-22]$ \\

$286$ & $[a-1,1,0,2,-1,2]$ & $ [-2,0,0,2,-2,2]$ & $[0, a, 0, 0, 0, a+20]$, $a\geq 1$ \\

$283$ & $[a-1,0,1,2,-1,2]$ & $ [-2,0,0,2,-2,2]$ & $[0, a+2, 0, 0, 0, a-2]$, $a\geq 1$ \\

$370$ &  $[1,b,1,1,1,1]$ &  $[1,-4,\frac{3}{2},0,\frac{3}{2},1]$ &  $[0,0,0,b+3,0,0]$\\

$272$ &  $[2,b-1,0,1,0,2]$ &  $[\frac{3}{2},-2,-\frac{1}{2},1,-\frac{1}{2},\frac{3}{2}]$ & $[1,1,0,b-1,2,0]$, $b\geq 1$ \\

$244$ &  $[1,b,1,1,1,1]$ &  $[\frac{5}{2},-\frac{5}{2},0,0,0,\frac{5}{2}]$ &  $[0,b+2,0,0,0,-3b-30]$\\

$243$ &  $[1,b,1,1,1,1]$ &  $[\frac{5}{2},-\frac{5}{2},0,0,0,\frac{5}{2}]$ &  $[b+2,0,0,0,0,3b+30]$\\

\cline{1-4}
$217$ & $[2,b-2,-1,3,-1,2]$ & $[1,-2,-1,2,-1,1]$ & $[1, 0, 0, b-1, 3, -3]$, $b\geq 1$ \\
& & & $[0, 1, 0, b-1, 3, 3]$, $b\geq 1$ \\
\cline{1-4}

$222$ & $[1,1,c,1,1,1]$ & $[1,2,-\frac{5}{2},0,0,2]$ & $[0,c+2,0,0,c+3,-c-8]$\\

$224$ & $[1,1,1,1,e,1]$ & $[2, 2,0,0,-\frac{5}{2},1]$ & $[e+2,0,0,0,e+3,e+8]$\\

$435$ & $[1,1,1,1,1, f]$ & $[0,\frac{3}{2},2,0,\frac{3}{2},-5]$ & $[ f+3, 0, 0, 0, 0, -f-11]$ \\

$371$ & $[1,1,1,1,1, f]$ & $[1,0,3,0,0,-\frac{7}{2}]$ & $[0, 0, 0, 0, f+3, 2f+22]$ \\

$293$ & $[2,1,-1,2,0,f-1]$ & $[2,0,-2,2,0,-2]$ & $[f, 0, 0, 0, 0, -f-20]$, $f\geq 1$ \\

$288$ & $[2,0,-1,2,1,f-1]$ & $[2,0,-2,2,0,-2]$ & $[f+2, 0, 0, 0, 0, -f+2]$, $f\geq 1$ \\
\Xhline{0.8pt}
\end{tabular}
\label{table-EIII-string-part-card5-part2}
\end{table}

\medskip
\centerline{\scshape Funding}
Dong was supported by NSFC grant 11571097 (2016-2019). He is partially supported by NSFC grant 11771272.

\medskip
\centerline{\scshape Acknowledgements}
We thank Dr.~Bai sincerely for guiding us through Section \ref{sec-exam}, and Daniel Wong for helping us with Remark \ref{rmk-Wallach}.
The authors are deeply grateful to the \texttt{atlas} mathematicians for many things.

\end{document}